\documentclass[a4paper,11pt,fancyhdr]{article}

\usepackage{amsmath,amssymb,graphicx,fancyhdr,indentfirst,%
  undertilde,color,multirow,slashed,mathtools,listings,cite}
\usepackage[english]{babel}
\usepackage[colorlinks=true,urlcolor=black]{hyperref}
\usepackage[makeroom]{cancel}
\usepackage{amsthm}
\usepackage{geometry}
\usepackage{tikz}
\usepackage{graphicx}
\usepackage{textcomp}
\usepackage{tikz-cd}
\usepackage{appendix}
\usepackage{subcaption}
\usetikzlibrary{decorations.markings}
\usepackage[ruled,vlined]{algorithm2e}
\usepackage{mathrsfs}
\usepackage{authblk}
\usetikzlibrary{arrows}

\definecolor{lightgrey}{gray}{0.9}

\newcommand{\expect}[1]{\mathbb{E} \!\left [ #1\right]}

\newcommand{\abs}[1]{\left\lvert #1 \right\rvert}

\newcommand{\ceil}[1]{\left\lceil #1 \right\rceil}
\newcommand{\ora}{\overrightarrow}
\newcommand{\ola}{\overleftarrow}

\newcommand{\innermid}{\nonscript\;\delimsize\vert\nonscript\;}
\newcommand{\activatebar}{%
  \begingroup\lccode`\~=`\|
  \lowercase{\endgroup\let~}\innermid 
  \mathcode`|=\string"8000
}

\newcommand{\VV}{\mathbb{V}}

\newcommand{\Sphere}{\mathbb{S}}

\newcommand{\N}{\mathbb{N}}

\newcommand{\R}{\mathbb{R}}
\newcommand{\Z}{\mathbb{Z}}
\newcommand{\PP}{\mathbb{P}}

\newcommand{\NN}{\mathcal{N}}

\newcommand{\bcode}{\mathcal{B}}

\newcommand{\Linfty}{L^{\infty}}
\newcommand{\pvar}{p-\text{var}}

\newcommand{\suchthat}{\;\ifnum\currentgrouptype=16 \middle\fi|\;}

\DeclareMathOperator\id{id}
\DeclareMathOperator\Vol{Vol}

\DeclareMathOperator\PH{PH}
\DeclareMathOperator\Pers{Pers}

\DeclareMathOperator\Homeo{Homeo}

\DeclareMathOperator\Lip{Lip}

\DeclareMathOperator\Image{Im}

\DeclareMathOperator\rk{rank}

\DeclareMathOperator\Dgm{Dgm}

\DeclareMathOperator\Alg{Alg}
\DeclareMathOperator\OT{OT}
\DeclareMathOperator\updim{\overline{dim}}
\DeclareMathOperator\downdim{\underline{dim}}

\newcommand{\tensor}{\otimes}
\newcommand{\ds}{\oplus}
\newcommand{\DS}{\bigoplus}
\newcommand{\al}{\alpha}

\newcommand{\veps}{\varepsilon}
\newcommand{\vp}{\varphi}

\newcommand{\bd}{\begin{displaymath}
\begin{tikzcd}}
\newcommand{\ed}{\end{tikzcd}
\end{displaymath}}
\newcommand{\bmat}{\begin{pmatrix}}
\newcommand{\emat}{\end{pmatrix}}
\newcommand{\be}{\begin{equation}}
\newcommand{\ee}{\end{equation}}
\newcommand{\btikz}{\begin{tikzcd}}
\newcommand{\etikz}{\end{tikzcd}}
\newcommand{\bea}{\begin{eqnarray}}
\newcommand{\eea}{\end{eqnarray}}
\newcommand{\bse}{\begin{subequations}}
\newcommand{\ese}{\end{subequations}}
\newcommand{\bc}{\begin{center}}
\newcommand{\ec}{\end{center}}

\newcommand{\nonum}{\nonumber}

\newcommand{\half}{\frac{1}{2}}

\newcommand{\XX}{\mathcal{X}}

\DeclareMathOperator*{\argmax}{arg\,max}

\newcommand{\norm}[1]{\left\lVert#1\right\rVert}

\newcommand{\del}{\partial}

\newcommand{\xx}{{\bf{x}}}

\newcommand{\DD}{\mathcal{D}}

\newcommand{\comment}[1]{}

\newcommand{\ie}{{\it i.e. }}
\newcommand{\cf}{{\it cf. }}
\newcommand{\eg}{{\it e.g. }}

\newcommand{\Ham}{{\mathcal{H}}}
\newcommand{\Lag}{{\mathcal{L}}}
\newcommand{\Mel}{{\mathcal{M}}}

\def\blob[#1]{~\parbox{#1mm}{
\begin{fmfgraph*}(#1,#1)
\fmfleft{i1}
\fmfright{o1}
\fmf{phantom}{i1,v1,o1}
\fmfblob{0.4w}{v1}
\end{fmfgraph*}}~}
\def\vertex[#1]{~\parbox{#1mm}{
  \begin{fmfgraph*}(#1,#1)
    \fmfleft{i1, i2}
    \fmfright{o1,o2}
    \fmf{phantom}{i1,i2,o2,o1,i1}
    \fmf{plain}{i1,v,i2}
    \fmf{plain}{o1,v,o2}
    \fmfforce{nw}{i1}
    \fmfforce{sw}{i2}
    \fmfforce{se}{o1}
    \fmfforce{ne}{o2}
    \fmfforce{c}{v}
    \fmfdot{v}
  \end{fmfgraph*}
  }~}
\def\othervertex[#1]{~\parbox{#1mm}{
  \begin{fmfgraph*}(#1,#1)
    \fmfleft{i1, i2}
    \fmfright{o1,o2}
    \fmf{phantom}{i1,i2,o2,o1,i1}
    \fmf{plain}{i1,v}
    \fmf{plain}{i2,v}
    \fmf{plain}{o1,v}
    \fmf{plain}{v,o2}
    \fmfforce{nw}{i1}
    \fmfforce{sw}{i2}
    \fmfforce{se}{o1}
    \fmfforce{ne}{o2}
    \fmfforce{c}{v}
    \fmfblob{0.3w}{v}
  \end{fmfgraph*}
  }~}
\def\bigvertex[#1]{~\parbox{#1mm}{
  \begin{fmfgraph*}(#1,#1)
    \fmfleft{i1, i2}
    \fmfright{o1,o2}
    \fmf{phantom}{i1,i2,o2,o1,i1}
    \fmf{plain}{i1,v,i2}
    \fmf{plain}{o1,v,o2}
    \fmfforce{nw}{i1}
    \fmfforce{sw}{i2}
    \fmfforce{se}{o1}
    \fmfforce{ne}{o2}
    \fmfforce{c}{v}
    \fmfdot{v}
  \end{fmfgraph*}
  }~}
\def\edge[#1]{~\parbox{#1mm}{
  \begin{fmfgraph*}(#1,#1)
    \fmfleft{i}
    \fmfright{o}
    \fmf{plain,l.s=left}{i,o}
  \end{fmfgraph*}
}~}
\def\curlyedge[#1]{~\parbox{#1mm}{
  \begin{fmfgraph*}(#1,#1)
    \fmfleft{i}
    \fmfright{o}
    \fmf{curly,l.s=left}{i,o}
  \end{fmfgraph*}
}~}
\def\wavyedge[#1]{~\parbox{#1mm}{
  \begin{fmfgraph*}(#1,#1)
    \fmfleft{i}
    \fmfright{o}
    \fmf{wiggly,l.s=left}{i,o}
  \end{fmfgraph*}
}~}
\def\tadpole[#1]{~\parbox{#1mm}{
 	\begin{fmfgraph*}(#1,#1)
 		\fmfleft{i1}
 		\fmfright{o1}
		\fmf{plain}{i1,v1,v1,o1}
 		\fmfdot{v1}
 	\end{fmfgraph*}}~}
\def\amputatedtadpole[#1]{\begin{fmfgraph*}(#1,2)
		\fmfleft{i1}
		\fmfright{o1}
		\fmf{phantom}{i1,v1,o1}
		\fmf{plain}{v1,v1}
		\fmfdot{v1}
	\end{fmfgraph*}}
\def\amputatedwigglytadpole[#1]{\begin{fmfgraph*}(#1,2)
		\fmfleft{i1}
		\fmfright{o1}
		\fmf{phantom}{i1,v1,o1}
		\fmf{wiggly}{v1,v1}
		\fmfdot{v1}
	\end{fmfgraph*}}
\def\amputatedcurlytadpole[#1]{\begin{fmfgraph*}(#1,2)
		\fmfleft{i1}
		\fmfright{o1}
		\fmf{phantom}{i1,v1,o1}
		\fmf{curly}{v1,v1}
		\fmfdot{v1}
	\end{fmfgraph*}}
\def\vacfirstord[#1]{~\parbox{#1mm}{
 	\begin{fmfgraph*}(#1,#1)
		\fmfleft{i1,i2}
		\fmfright{o1,o2}
		\fmf{plain,left}{v2,v1}
		\fmf{plain,left}{v1,v2}
		\fmf{plain,left}{v3,v1}
		\fmf{plain,left}{v1,v3}
		\fmf{phantom}{v2,i1}
		\fmf{phantom}{v3,o1}
		\fmfforce{(0,0)}{i1}
		\fmfforce{(w,0)}{o1}
		\fmfforce{(0.5w,0.5h)}{v1}
		\fmfforce{(0,0.5h)}{v2}
		\fmfforce{(w,0.5h)}{v3}
		\fmfforce{nw}{i2}
		\fmfforce{ne}{o2}
		\fmfforce{(0,0.5h)}{i2}
		\fmfforce{(w,0.5h)}{o2}
		\fmfdot{v1}
		\end{fmfgraph*}}~}
\def\doubletadpolehor[#1]{~\parbox{#1mm}{
\begin{fmfgraph}(#1,#1)
	\fmfleft{i1}
	\fmfright{o1}
	\fmf{plain}{i1,v1,v1,v2,v2,o1}
	\fmfdot{v1,v2}
\end{fmfgraph}
}~}
\def\tripletadpolehor[#1]{~\parbox{#1mm}{
\begin{fmfgraph}(#1,#1)
	\fmfleft{i1}
	\fmfright{o1}
	\fmf{plain}{i1,v1,v1,v2,v2,v3,v3,o1}
	\fmfdot{v1,v2,v3}
\end{fmfgraph}
}~}
\def\doubletadpolever[#1]{~\parbox{#1mm}{
\begin{fmfgraph}(#1,#1)
	\fmfleft{i1}
	\fmfright{o1}
	\fmf{plain}{i1,v1}
	\fmf{plain,left}{v1,v2}
	\fmf{plain,left}{v2,v1}
	\fmf{plain}{v1,o1}
	\fmffreeze
	\fmf{plain,left}{v2,v3,v2}
	\fmfforce{c}{v2}
	\fmfforce{(0.5w,h)}{v3}
	\fmfforce{(0.5w,0)}{v1}
	\fmfforce{sw}{i1}
	\fmfforce{se}{o1}
	\fmfdot{v1,v2}
\end{fmfgraph}
}~}
\def\amputateddoubletadpolever[#1]{~\parbox{#1mm}{
	\begin{fmfgraph}(#1,#1)
	\fmfleft{i1}
	\fmfright{o1}
	\fmf{phantom}{i1,v1}
	\fmf{plain,left}{v1,v2}
	\fmf{plain,left}{v2,v1}
	\fmf{phantom}{v1,o1}
	\fmffreeze
	\fmf{plain,left}{v2,v3,v2}
	\fmfforce{c}{v2}
	\fmfforce{(0.5w,h)}{v3}
	\fmfforce{(0.5w,0)}{v1}
	\fmfforce{sw}{i1}
	\fmfforce{se}{o1}
	\fmfdot{v1,v2}
	\end{fmfgraph}~
}
}
\def\sunset[#1]{\parbox{#1mm}{
\begin{fmfgraph}(#1,#1)
\fmfleft{i}
\fmfright{o}
\fmfforce{0,0.5h}{i}
\fmfforce{w,0.5h}{o}
\fmf{plain,tension=5}{i,v1}
\fmf{plain,tension=5}{v2,o}
\fmf{plain,left,tension=0.5}{v1,v2,v1}
\fmf{plain}{v1,v2}
\fmfdot{v1,v2}
\end{fmfgraph}
}}
\def\amputatedsunset[#1]{\parbox{#1mm}{
\begin{fmfgraph}(#1,#1)
\fmfleft{i}
\fmfright{o}
\fmf{phantom,tension=5}{i,v1}
\fmf{phantom,tension=5}{v2,o}
\fmf{plain,left,tension=0.5}{v1,v2,v1}
\fmf{plain}{v1,v2}
\fmfdot{v1,v2}
\end{fmfgraph}
}}
\def\fourpointsecondorder[#1]{~\parbox{#1mm}{
	\begin{fmfgraph}(15,#1)
	\fmfleft{i1,i2}
	\fmfright{o1,o2}
	\fmf{plain}{i1,v1}
	\fmf{plain}{i2,v1}
	\fmf{plain,right}{v1,v2}
	\fmf{plain,left}{v1,v2}
	\fmf{plain}{v2,o1}
	\fmf{plain}{v2,o2}
	\fmfdot{v1,v2}
	\end{fmfgraph}}~
}
\def\fourpointsecondordertwo[#1]{~\parbox{#1mm}{
	\begin{fmfgraph}(15,#1)
	\fmfleft{i1,i2,i3}
	\fmfright{o1}
	\fmf{plain}{i1,v1}
	\fmf{plain}{i2,v1}
	\fmf{plain}{i3,v1}
	\fmf{plain}{v1,v2,v2,o1}
	\fmfdot{v1,v2}
	\end{fmfgraph}
}~}

\newtheorem{theorem}{Theorem}[section]

\catcode`,\active

\catcode`\,12

\theoremstyle{definition}
\newtheorem{definition}[theorem]{Definition}

\newtheorem{notation}[theorem]{Notation}
\newtheorem{lemma}[theorem]{Lemma}
\newtheorem{proposition}[theorem]{Proposition}

\theoremstyle{remark}
\newtheorem{remark}[theorem]{Remark}
\theoremstyle{example}

\newtheorem{corollary}[theorem]{Corollary}

\title{On $C^0$-persistent homology and trees}

\author[1,2,3]{Daniel Perez\thanks{Email: \texttt{daniel.perez@ens.fr}}}
\affil[1]{\footnotesize D\'epartement de math\'ematiques et applications, \'Ecole normale sup\'erieure, CNRS, PSL University, 75005 Paris, France}
\affil[2]{\footnotesize Laboratoire de math\'ematiques d'Orsay, Universit\'e Paris-Saclay, CNRS, 91405 Orsay, France}
\affil[3]{\footnotesize DataShape, Centre Inria Saclay, 91120 Palaiseau, France}

\date{\today}

\begin{document}
\maketitle

\begin{abstract}
In this paper we give a metric construction of a tree which correctly identifies connected components of superlevel sets of continuous functions $f:X\to \R$ and show that it is possible to retrieve the $H_0$-persistent diagram from this tree. We revisit the notion of homological dimension previously introduced by Schweinhart and give some bounds for the latter in terms of the upper-box dimension of $X$, thereby partially answering a question of the same author. We prove a quantitative version of the Wasserstein stability theorem valid for regular enough $X$ and $\al$-H\"older functions and discuss some applications of this theory to random fields and the topology of their superlevel sets. 
\end{abstract}
\tableofcontents

\section{Introduction}

\subsection{State of the art}
The topology of superlevel sets of a function has been a widespread subject of study in different mathematical communities. In the probability theory community, the introduction of trees has key in the understanding of connected components of superlevel sets of random functions on $[0,1]$ \cite{LeGall:Trees,Picard:Trees,curien2013,LeGallDuquesne:LevyTrees}. This approach allows us to define trees associated to (arbitrarily irregular, but continuous) functions.

More recently, so-called \textit{merge trees} have made their apparition amongst the topological data analysis (TDA) commmunity (\cf the books by Chazal \textit{et al.} \cite{Chazal:Persistence} and Oudot's book \cite{Oudot:Persistence} for an introduction to TDA). As in the probabilistic case, these merge trees carry important information about the connected components of superlevel sets and moreover about the persistence diagram of a function $f$ defined on a compact space $X$ \cite{Curry_2018,Curry_2021,Wang_2014,Memoli_2020,Munch_2019}, which is now required to be a Morse function (an explicit construction and correspondence between trees and barcodes can be found in \cite{Curry_2018}). 

The construction of these trees are different between both communities: the approach of the probabilists is analytic \cite{LeGall:Trees,curien2013}, whereas the merge trees can be seen as an algebraic construction \cite{Curry_2018,Munch_2019,Curry_2021}. Since these trees capture essentially the same information about the connected components of superlevel sets, one can ask whether both constructions coincide where their regimes of validity intersect. We will show in this paper that they do and that it is possible to retrieve the $H_0$-persistence diagram of $f$ from the constructed tree (constructed from through the probabilistic approach). 

Parallel to this development, Wasserstein $p$ distances on the space of diagrams (denoted $d_p$) \cite[Chapter VIII.2]{HarerEdelsbrunner:CT} have been widely used and studied by the TDA community in different contexts \cite{Carriere_2017, Divol_2019,Turner_2014,LipschitzStableLpPers,Mileyko_2011}. Recently, Wasserstein distances have been formalized through the use of optimal \textit{partial} transport by Divol and Lacombe \cite{Divol_2019}. In this approach, we look at persistence diagrams as measures, a point of view which had been previously been introduced \cite{Oudot:Persistence,Chazal:Persistence} and has proved fruitful independently from these considerations. The framework introduced by Divol and Lacombe extends the notion of Wasserstein $p$ distance previously defined on persistence diagrams to arbitrary Radon measures on the upper-half plane $\XX$ where persistence diagrams are defined. 

The extension to all Radon measures comes with certain advantages, such as having an easily definable and computable notion of ``average diagram'', defined by duality. This notion was originally introduced by Chazal and Divol in \cite{ChazalDivol:Brownian} as follows. If $f$ is a random function, seeing $\Dgm(f)$ as a measure, it is possible to define the average diagram of the process by duality in the following way. For every measurable set $B \subset \XX$, 
\be
\expect{\Dgm(f)}(B) := \expect{\Dgm(f)(B)} \,.
\ee
From the definition, $\expect{\Dgm(f)}$ encodes every linear functional of the diagram and is easily computed, motivating its introduction. Note this definition contrasts the Fréchet means approach of other authors (\eg Turner \textit{et al.} \cite{Turner_2014}), which is non-linear, depends on $p$ and requires a proof of existence and unicity, but does not require the extension of the space of persistence diagrams to the space of arbitrary measures on $\XX$.

This dual approach of Chazal and Divol inscribes itself in a more general interest in the persistence diagrams of stochastic processes, which have been studied by a wide variety of authors, for instance \cite{AdlerTaylor:RandomFields,Baryshnikov_2019,Adler_2010,ChazalDivol:Brownian,Chazal_2015,Chazal_2014,Perez_Pr_2020, Turner_2014,Adams_2020}. Some of the previously cited results discuss different aspects of random field persistence theory, which include, but are not limited to computations for canonical processes \cite{Baryshnikov_2019,Perez_Pr_2020}, stability of certain linear functionals with respect to the bottleneck distance \cite{Chazal_2015}, the Euler characteristic \cite{Adler_2010}, random complexes \cite{Adler_2010} and notions of central tendency \cite{ChazalDivol:Brownian, Turner_2014}.

Given the widespread use of Wasserstein $p$ distances, it is important to understand whether this notion is continuous (and the nature of this continuity) with respect to perturbations at the level of the filtrating functions on the space $X$. This so-called ``Wasserstein stability'' of persistence diagrams of functions $f: X \to \R$ has been widely discussed by the TDA community, in the context where the space $X$ is triangulable. There are many results in this direction \cite{Chen_2011,Skraba_2020,LipschitzStableLpPers}, valid with different degrees of generality, covering both $X$ compact \cite{LipschitzStableLpPers,Skraba_2020} and $X$ non-compact \cite{Chen_2011}, but mainly focusing mainly on Lipschitz functions (note, however, that the work of Chen and Edelsbrunner \cite{Chen_2011} does not require the Lipschitz condition). The first result in this direction was obtained by Cohen-Steiner \textit{et al.} \cite{LipschitzStableLpPers} and depends on the following restriction on $X$.
\begin{definition}{\cite{LipschitzStableLpPers}}
A (triangulable) metric space \textbf{$X$ implies bounded $q$-total persistence} if, for all $k \in \N$, there exists a constant $C_X$ that depends only on $X$ such that
\be
\Pers_q^q(\Dgm_k(f)) < C_X 
\ee
for every tame function $f$ with Lipschitz constant $\Lip(f) \leq 1$.
\end{definition}
The $\Pers_p$-functional of the definition above is the usual $p$-persistence used in TDA (a non-exhaustive list of uses of this functional includes \cite{Carriere_2017,Adams_2017,Divol_2019,Turner_2014,Mileyko_2011}), defined as the $\ell^p$-norm of the length of the bars of the barcode of $f$. 
The results obtained thereafter rely heavily on this condition, which is not rendered quantitative (in particular, given $X$, no upper bound for $C_X$ or lower bound for $q$ were known in general). Nonetheless, this condition allowed the authors to show Wasserstein stability,
\begin{theorem}[Cohen-Steiner, Edelsbrunner, Harer, \cite{LipschitzStableLpPers}]
Let $X$ be a triangulable space implying bounded $q$-total persistence and let $f$ and $g$ be two $\R$-valued Lipschitz functions on $X$. Then, for all $p>q$, we have
\be
d_p(\Dgm(f),\Dgm(g)) \leq C_X (\Lip(f)^q \vee \Lip(g)^q) \norm{f-g}_\infty^{1-\frac{q}{p}} \,,
\ee
where $\Lip(f)$ denotes the Lipschitz constant of $f$.
\end{theorem}
Further results in this direction, such as \cite{Skraba_2020}, also rely on the bounded $q$-total persistence condition, but give bounds lower bounds on admissible $q$, finding that $q\geq d$, where $d$ is the maximal dimension of simplices in the triangulation of $X$. It is also known that, for distance functions to point clouds in $\R^d$, $q=d$.

We will later see that the lower bound for the validity of Wasserstein stability is closely related to a different question regarding the link between the so-called \textit{homological dimensions} of $X$ and the upper-box dimension of $X$, which we will denote $\updim(X)$ (analogously, we will denote $\downdim(X)$ the lower-box dimension). To the best knowledge of the author, although Yuliy Baryshnikov and Shmuel Weinberger had previously obtained results in this direction (but never published them), this question was first opened and studied by Schweinhart and MacPherson \cite{MacPherson_2012} and later studied in more detail by Schweinhart in \cite{Schweinhart_2019}, but has also been addressed by other authors (\cf \cite{Adams_2020} and the references therein). 
\begin{definition}[Schweinhart's definition of $\PH_k$, \cite{Schweinhart_2019}]
Let $X$ be a bounded subset of a metric space. The $\PH_k$-dimension of $X$ is
\be
\dim_{\PH}^k(X):= \sup_{\xx} \inf_p \; \{\Pers_p(\Dgm_k(d(-,\xx))) <\infty\} \,,
\ee
where the supremum is taken over all finite sets of points $\xx$ of $X$.
\end{definition}
There are open problems stated in Schweinhart's paper regarding the relation between these notions of dimension and $\updim(X)$, some of which we will give a partial answer to in this paper. As we will later see, it is suitable to tweak this definition slightly.

\subsection{Our contribution}
\begin{itemize}
  \item Following the work of Le Gall and Curien \cite{LeGall:Trees,curien2013}, we define a tree constructed from a compact, connected and locally path connected space $X$ and a continuous function $f: X \to \R$ using a pseudo-distance on $X$ defined in terms of $f$ (section \ref{sec:H0treeconstruction}). We use this constructed tree to extend the work of Curry \cite{Curry_2018} previously valid under a Morse assumption to every continuous function. More precisely, we prove that it is possible to retrieve the barcode of the function from the constructed tree via an explicit algorithm (theorem \ref{thm:TfcalulatesH0}).
  \item We show that the map assigning a function $f : [0,1] \to \R$ to its constructed tree $T_f$ is a surjection onto the space of trees of finite upper-box dimension and provide an explicit construction of an inverse image (section \ref{sec:inverseproblem}).
  \item Following previous work by Picard \cite{Picard:Trees} and Schweinhart \cite{Schweinhart_2019}, we introduce the so-called \textit{persistence index of degree $k$}, $\Lag_k(f)$ of a function $f:X\to \R$ (definition \ref{def:PersIndex}).  For a regular enough metric spaces $X$, if $f$ is H\"older continuous (or H\"older continuous up to precomposition by a homeomorphism), we show an upper bound for $\Lag_0(f)$ in terms of $\updim(X)$ (lemma \ref{lemma:RegularityDimension}) and show that $\Lag_0(f) = \updim(T_f)$, where $T_f$ denotes the tree constructed from $f$ (theorem \ref{thm:Lfandupboxdim}). 
  \item We modify Schweinhart's definition for the $k$th degree homological dimension of $X$ (definition \ref{def:HomDim}) as 
  \be
  \dim^k_{\PH}(X) := \sup_{f \in \Lip_1(X)} \Lag_k(f) \,
  \ee
  and show that under for $X$ regular enough, we retrieve a well-known result by Kozma \textit{et al.} that $\dim^0_{\PH}(X)=\updim(X)$ \cite{Kozma_2006} and moreover that the supremum in this definition is attained generically (theorem \ref{thm:genericityofHomDim}). Moreover, $\dim_{\PH}^k(X)$ can be bounded above by $\updim(X)$, up to a factor which may depend on $k$ and the regularity of $X$ (theorem \ref{thm:genericityTotalHomology}). 
  \item We show that the supremum in the definition of $\dim^k_{\PH}$ could have been taken over any regularity class $C^\al(X,\R)$, up to a factor of $\al$ (theorems \ref{thm:genericityofHomDim} and \ref{thm:genericityTotalHomology}) and show a genericity result for the set of functions in $C^\al(X,\R)$ satisfying $\updim(X) = \Lag_0(f)$. Moreover, under more stringent conditions on $X$, we show the same genericity result holds in fact for $\updim(X) \leq \Lag_k(f)$ for integer $0\leq k <\updim(X)$ and show the equality case with some supplementary conditions on $X$ and in particular for compact Riemannian manifolds (theorem \ref{thm:SchweinhartAns}). In so doing, we answer a question by Schweinhart \cite{Schweinhart_2019} regarding bounds on homological dimensions and regularity conditions on $X$ for this bound to be sharp (section \ref{sec:Schweinhart}).
  \item Using the results relating to the bounds on the homological dimensions, we give a Wasserstein stability result valid for all degrees of \v{C}ech homology on regular enough metric spaces (which in particular include all compact smooth manifolds of convexity radius bounded below) (theorem \ref{thm:WassersteinpStabilityLLC}), for which explicit bounds on the constant $C_X$ are given and sharp bounds on the regime of validity of the theorem (corollary \ref{cor:BoundOnCX}). We show an annex result of stability for the trees constructed from the functions $f$ in terms of the Gromov-Hausdorff distance between the trees (theorem \ref{thm:diagstreesqiso}).
  \item Finally, we discuss some easy consequences of these results applied to the stochastic setting (section \ref{sec:StochasticProcesses}) and prove Chazal \textit{et al.}-like results \cite{Chazal_2014} for the $d_p$-stability of average diagrams of stochastic processes with an \textit{a priori} hypothesis of regularity (theorem \ref{thm:randomfielddiagStability}).
\end{itemize}

\section{Barcodes, diagrams and trees}

\subsection{Trees stemming from a continuous function}
\label{sec:H0treeconstruction}
Unless otherwise specified, throughout this section, let $X$ denote a connected, locally path-connected, compact topological space and let $f: X \to \R$ be a continuous function. Let us denote $(X_r)_{r \in \R}$ the filtration of $X$ by the \textbf{superlevels} of $f$, that is
\be
X_r := \{x \in X \, \vert \, f(x) \geq r\} \,.
\ee
\begin{notation}
We will denote the open superlevel sets by $X_{>r}$ whenever necessary and note $X_{r}^z$ the connected component of $X_r$ containing $z$.
\end{notation}
There exists a pseudo-distance on $X$, denoted $d_f$, given by:
\begin{definition}
Let $X$ and $f$ be defined as above. The \textbf{$H_0$-distance}, $d_f$, is the pseudo-distance
\be
d_f(x,y) := f(x) + f(y) - 2 \sup_{\gamma: x \mapsto y} \inf_{t \in [0,1]} f(\gamma(t))\,,
\ee 
where the supremum runs over every path $\gamma$ linking $x$ to $y$.
\label{def:explicitpseudodistance}
\end{definition}
\begin{remark}
Notice there are different ways of writing this distance. In particular, the sup above is also characterized by
\begin{align}
\sup_{\gamma: x \mapsto y} \inf_{t \in [0,1]} f(\gamma(t)) &= \sup \{r \, \vert \, [x]_{H_0(X_r)} = [y]_{H_0(X_r)}\} \\
&= \sup\{r \, \vert \, \exists \gamma \in C_1(X_r) \, \text{ such that }\, \del \gamma = x-y\}\,.
\end{align}
These equalities hold, since we take the coefficients of homology with respect to $\Z/2\Z$, so we can interpret 1-cycles as sums of paths on $X$.
\end{remark}
This pseudo-distance is a generalization of the distance introduced by Curien, Le Gall and Miermont in \cite{curien2013}. Note that $d_f$ has the following properties:
\begin{enumerate}
  \item \textbf{Identification of the connected components of superlevel sets}: $d_f(x,y) = 0$ if and only if there exists $t \in \R$ such that $x, y \in \{f=t\} \text{ and for every } \veps >0$, $x$ and $y$ lie in the same connected component of $X_{>t-\veps}$;
  \item \textbf{Compatibility with the filtration induced by $f$}: Let $x,y \in X$ and suppose that $f(x)<f(y)$, then if $[x]_{H_0(X_{f(x)})} = [y]_{H_0(X_{f(x)})}$, 
    \be
    d_f(x,y) := \abs{f(x)-f(y)} \,.
    \ee 
\end{enumerate}
The compatibility with the filtration induced by $f$ is immediate from the definition of $d_f$. It remains to show the two following propositions.
\begin{proposition}
The function $d_f : X^2 \to \R^+$ of definition \ref{def:explicitpseudodistance} is indeed a pseudo-distance.
\end{proposition}
\begin{proof}
Checking symmetry and positivity is easy. The only non-obvious point is that the triangle inequality is satisfied by this expression. Let $x,y,z \in X$ and denote 
\be
[x\mapsto y] := \sup_{\gamma: x \mapsto y} \inf_{t \in [0,1]} f \circ \gamma(t) \,.
\ee
It suffices to show the following inequality 
\be
[x \mapsto z] + [z \mapsto y] \leq [x \mapsto y] + f(z) \,.
\ee
Let $\gamma$ be a path from $x$ to $z$ and $\eta$ be a path from $z$ to $y$ and let $\gamma * \eta$ be the concatenation of these two paths. By definition,
\be
\inf_{t\in [0,1]} f\circ(\gamma * \eta)(t) \leq [x \mapsto y] \,,
\ee
from which it follows that
\be
[x \mapsto z] \wedge [z \mapsto y] \leq [x \mapsto y]\,.
\ee
Without loss of generality, suppose that $[x \mapsto z]$ achieves the above minimum and note that
\be
 [z \mapsto y] \leq f(z)
\ee
by definition of $[z \mapsto y]$. Adding the two last inequalities together,
\be
[x \mapsto z] + [z \mapsto y] \leq [x \mapsto y] + f(z)\,,
\ee
as desired.
\end{proof}
\begin{proposition}
\label{prop:dfidentifiescc}
Let $f$ be a continuous function as above, then $d_f$ identifies the connected components of the superlevel sets.
\end{proposition}
\begin{proof}
The $(\Leftarrow)$ direction is immediate, so let us show $(\Rightarrow)$. 
\par
Suppose that $d_f(x,y) = 0$ and that $f(x)\neq f(y)$, then, 
\be
\sup_{\gamma: x \mapsto y} \inf_{t \in [0,1]} f(\gamma(t)) = \frac{f(x)+f(y)}{2} > f(x) \wedge f(y) \,.
\ee
However,
\be
\sup_{\gamma: x \mapsto y} \inf_{t \in [0,1]} f(\gamma(t))  \leq f(x) \wedge f(y) \,,
\ee
which leads to a contradiction, so $f(x) = f(y)$. The condition $d_f(x,y)=0$ becomes:
\be
f(x) = \sup_{\gamma:x \mapsto y} \inf_{t\in [0,1]} f(\gamma(t)) \,.
\ee
This is only possible if for every $\veps >0$ there is a path $\gamma$ lying entirely in $X_{>f(x)-\veps}^x$, so
\be
x,y \in \bigcap_{\veps>0} X_{>f(x)-\veps}^x
\ee
finishing the proof.
\end{proof}
With these technicalities out of the way, let us consider the metric space
\be
(T_f,d_f):= (X/\{d_f=0\},d_f)\;,
\ee
where $X/\{d_f=0\}$ denotes the quotient of $X$ where we identify all points $x$ and $y$ on $X$ satisfying $d_f(x,y)=0$. Slightly abusing the notation, let $d_f$ denote the distance induced on $T_f$ by the pseudo-distance $d_f$ on $X$.
\par
The metric structure of $T_f$ turns out to be simple, as $T_f$ is an $\R$-tree. Let us briefly recall the definition of an $\R$-tree.
\begin{definition}[Chiswell, \cite{Chiswell_2001}]
An \textbf{$\R$-tree} $(T, d)$ is a connected metric space such that any of the following equivalent conditions hold:
\begin{itemize}
  \item $T$ is a geodesic connected metric space and there is no subset of $T$ which is homeomorphic to the circle, $\Sphere_1$;
  \item $T$ is a geodesic connected metric space and the Gromov 4-point condition holds, \ie :
  \be
  \forall x,y,z,t \in T \quad d(x,y) + d(z,t) \leq \max\left[d(x,z) +d(y,t), d(x,t)+d(y,z)\right] \,; \nonum
  \ee
  \item $T$ is a geodesic connected $0$-hyperbolic space.
\end{itemize}
A \textbf{rooted $\R$-tree $(T,O,d)$} is an $\R$-tree along with a marked point $O \in T$.  
\end{definition}
A first important remark is that since $X$ is connected, so is $T_f$. To show $T_f$ is an $\R$-tree, we will use the first characterization of the definition above and show both conditions, \textit{i.e.} that there are no subspaces of $T_f$ which are homeomorphic to $\Sphere_1$ and that $T_f$ is in fact a geodesic metric space, to be satisfied separately. 
\par
Before showing this, it is helpful to introduce some notation. 
\begin{notation}
Let $\pi_f: X \to T_f$ denote the canonical projection onto $T_f$ and let $O$ denote the root of $T_f$ (\textit{i.e.} $f(O)=\min f$), let us define the following quantity
\begin{align}
\label{eq:elltau}
\ell(\tau) &:= \inf_X f + d_f(O,\tau) \,,
\end{align}
where $X^\tau_{f(\tau)}$ denotes the connected component of the superlevel set $X_{f(\tau)}$ containing a preimage of $\tau$.
\end{notation}
\begin{remark}
These objects are well-defined by definition of $d_f$. 
\end{remark}
\begin{definition}
The \textbf{pseudo-distance topology on $X$} or the \textbf{topology of $d_f$} is the topology on $X$ generated by the open balls: 
\be
B(x,r) := \{z \in X \, \vert \, d_f(x,z)<r\}
\ee
\end{definition}
Despite the fact that the pseudo-distance topology is not in general Hausdorff, it is nonetheless fine enough to be useful, as shown by the two following technical lemmas.
\begin{lemma}
\label{lemma:topodftopoususual}
Let $X_{>r}$ denote the open superlevel set $\{f>r\}$ on $X$, then $X_{>r}$ has same connected components for the topology of $d_f$ on $X$ and the usual topology of $X$.
\end{lemma}
\begin{proof}
Let us start by noticing that $X_{>r}$ is open in $X$ for both topologies. For the usual topology, it is trivial. For the topology of $d_f$ it is the complement in $X$ of the closed ball $\overline{B(p,r-\inf f)}$, where $p$ is a point on $X$ achieving the infimum of $f$, which exists by compactness of $X$.
\par
Let $Y$ now denote a connected component of $X_{>r}$ for the usual topology. The set $Y$ is connected for the topology of $d_f$. Otherwise, we could write $Y = U \sqcup V$ for some open sets $U$ and $V$, but since open sets of the topology of $d_f$ are also open for the usual topology, this leads to a contradiction, as we assumed $Y$ was connected for the usual topology. We will now show that $Y$ is both open and closed in $X_{>r}$ for the topology of $d_f$. $Y$ is open, since it can be written as the union of open balls
\be
Y = \bigcup_{y \in Y} B(y,f(y)-r) \,.
\ee
Additionally, $Y$ is closed since its complement is open, as it can be similarly written as the union of open balls. It follows that $Y$ is also a connected component of $X_{>r}$ for the topology of $d_f$.
\par
Now, suppose that $Y$ is a connected component of $X_{>r}$ for the topology of $d_f$. Any ball of the covering above is path connected, but since $Y$ is connected, this implies that $Y$ is path connected (and the paths are completely included within $Y$), it is thus a path connected component of $X_{>r}$. Since $X$ is connected and locally path connected for the usual topology, $Y$ is a path connected component of the usual topology, rendering it a connected component for the usual topology.
\end{proof}
\begin{lemma}
\label{lemma:pifbijectioncc}
Denote $T_{>r}$ the open superlevel set on $T_f$. For the topology of $d_f$, $\pi_f$ induces a bijective correspondence between the connected components of $X_{>r}$ and those of $T_{>r}$.
\end{lemma}
\begin{proof}
Since $\pi_f$ is surjective and is both open and closed for the topology of $d_f$, lemma \ref{lemma:topodftopoususual} implies that the map $\pi_f$ surjectively sends the connected components of $X_{>r}$ onto connected components of $T_{>r}$, since the connected components of $X_{>r}$ for the topology of $d_f$ and the usual topology of $X$ are the same. 
\par
It remains to show the injectivity. Note that $\pi_f$ is open and closed for the topology of $d_f$ on $X$. The connected components of $X_{>r}$ are either disjoint or equal and, in fact, so are the images by $\pi_f$ of these connected components. Otherwise, there exists some $\tau \in T_{>r}$ such that there is a preimage of $\tau$ lying in two different connected components of $X_{>r}$, which is impossible, as every preimage of $\tau$ must lie in the same connected component of $X_{>r}$ in accordance to proposition \ref{prop:dfidentifiescc}. This is equivalent to stating that if $Y$ and $Z$ are two connected components of $X_{>r}$ and $Y \neq Z$, then $\pi_f(Y) \cap \pi_f(Z) = \emptyset$, in particular $\pi_f(Y) \neq \pi_f(Z)$. Symbolically,
\be
Y \neq Z \Rightarrow \pi_f(Y) \neq \pi_f(Z) \,,
\ee 
which is the contrapositive of the statement of injectivity.  
\end{proof}

From the above lemmata, we get the following proposition.
\begin{proposition}
\label{prop:dfinducestrees}
The metric space $T_f := X/\{d_f = 0\}$ equipped with distance $d_f$ possesses no subspace homeomorphic to $\Sphere_1$.
\end{proposition}
\begin{proof}
We will reason by contradiction. Suppose that $T_f$ contains $U \subset T_f$ such that $U$ is homeomorphic to the circle, $\Sphere_1$. Note that $f$ descends to a function on $T_f$ which is not locally constant anywhere by definition of $d_f$ and in particular not locally constant anywhere on $U$, as the level-sets of $f$ in $T$ are totally discontinuous.
\begin{figure}[h!]
  \centering
    \includegraphics[width=0.6\textwidth]{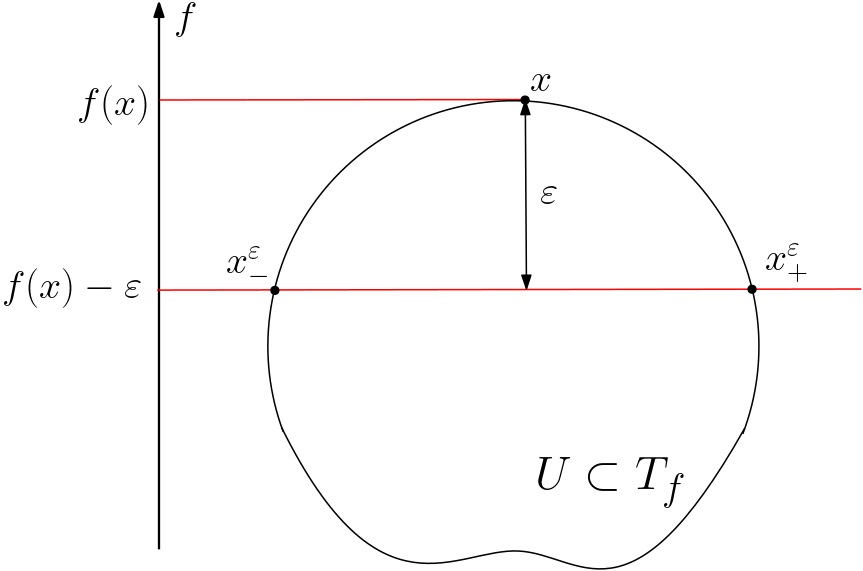}
\end{figure}
\par
It follows that there exists an element $x \in U$ such that the maximum of $f$ on $U$ is attained at $x$. For $\veps >0$ small enough, there are two distinct points $x_-^\veps$ and $x_+^\veps$ such that $f(x_+^\veps) = f(x)-\veps = f(x_-^\veps)$. Without loss of generality, we pick these points to be the closest ones to $x$ along an arbitrary parametrization of $U$ where this equality occurs. Since $U$ is homeomorphic to $\Sphere_1$, there is a path $\gamma$ linking $x_+$ and $x_-$ lying entirely above $f(x)-\veps$ and passing through $x$. The image of $\gamma$ in $T_{>f(x)-\veps}$ is contained within one and only one connected component of $T_{>f(x)-\veps}$, which we will denote $S$. By lemma \ref{lemma:pifbijectioncc}, $S$ corresponds to a unique connected component of $X_{>f(x)-\veps}$ with respect to the topology of $d_f$, which we will denote $X^S$. By lemma \ref{lemma:topodftopoususual}, $X^S$ is a connected component of $X_{>f(x)-\veps}$ for the usual topology. 
\par
For every $0<\veps' < \veps$, we can pick points $x_\pm^{\veps'}$ on $U$. The connected component $S$ contains $x_\pm^{\veps'}$ for every such $\veps'$ and since inverse images of these two points are connected in $X_{>f(x)-\veps}$, 
\be
d_f(x^{\veps'}_{+}, x^{\veps'}_- ) < 2 (f(x)-\veps') - 2 \!\!\!\!\!\inf_{\gamma :\, x^{\veps'}_{+} \mapsto x^{\veps'}_-} \!f < 2 (\veps- \veps') 
\ee
Letting $x_\pm^{\veps'} \to x_\pm^\veps$ in $U$ as $\veps'\to \veps$, we have that $d_f(x^\veps_-,x^\veps_+)=0$, leading to a contradiction, since we supposed that $x_+^\veps$ and $x_-^\veps$ were disjoint in $T_f$ (and therefore not a distance zero away from one-another).
\end{proof}

\subsection{From trees to barcodes}
Given a tree stemming from a continuous function $f: X \to \R$, it is possible to reconstruct the $H_0$-barcode of $f$ from $T_f$. If $T_f$ is finite, the relation between the barcode of $H_0(X,f)$ with respect to the superlevel filtration and the tree $T_f$ is given by algorithm \ref{alg:barcodefromtree}.

\begin{algorithm}[h!]
\SetAlgoLined
\KwResult{$\VV$}
 $\mathcal{F} \gets T$ \;
 $\VV \gets 0$ \;
 $i \gets 0$ \;
 \While{$\mathcal{F} \neq \varnothing$}{
  Find $\gamma$ the longest path in $\mathcal{F}$ starting from a root $\al$ and ending in a leaf $\beta$ \;
  \eIf{$i=0$}{
   $\VV \gets \VV \ds k[\ell(\al),\infty[$ \;
   }{
   $\VV \gets \VV \ds k[\ell(\al),\ell(\beta)[$ \;
  }
  $\mathcal{F} \gets \overline{\mathcal{F} \setminus \Image(\gamma)}$\;
  $i \gets i+1$ \;
 }
 \Return $\VV$
 \caption{A functorial relation between persistence modules and $\R$-trees}
\label{alg:barcodefromtree}
\end{algorithm}
\begin{figure}[h!]
  \centering
    \includegraphics[width=0.6\textwidth]{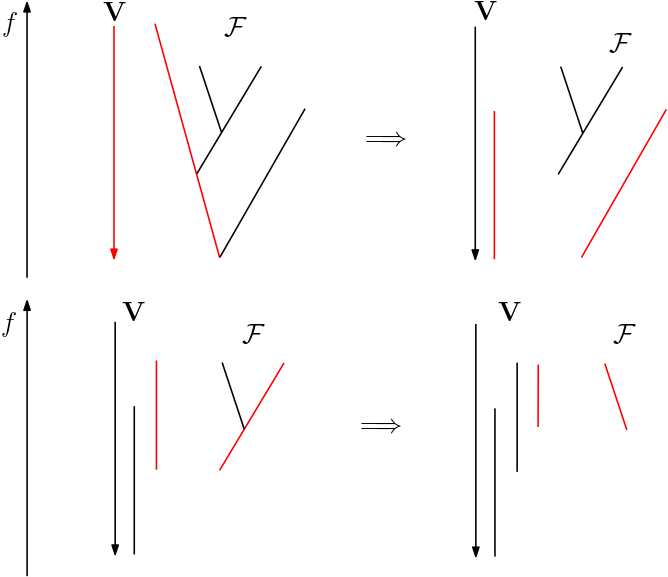}
  \caption{The first four iterations of algorithm \ref{alg:barcodefromtree}. For every step, in red is the longest branch of the tree, which we use to progressively construct the persistent module $\VV$ by associating an interval module whose ends correspond exactly to the values of the endpoints of the branches.}
  \label{fig:algorithm1}
\end{figure}
If $T_f$ is infinite, we can still give a correspondence between the barcode and the tree proceeding by approximation. This approximation procedure requires the introduction of so-called $\veps$-trimmings of $T_f$, of which we briefly recall the definition. Since the results of this section can be easily extended to any compact tree, we formulate the rest of this section in full generality. 
\par
For any rooted $\R$-tree $(T,d,O)$, we can define a filtering function $\ell: T \to \R$ by setting 
\be
\ell(\tau) := d(O,\tau) \,.
\ee 
This allows us to define the height above a point $\tau$ as follows.
\begin{definition}
The function of \textbf{the height above $\tau$} on a rooted $\R$-tree $T$ is a function $h: T \mapsto \R$, defined as
\be
h(\tau) := \sup_{\eta \in T^\tau_{\ell(\tau)}} d(O,\eta) - \ell(\tau)\,,
\ee
\end{definition}
The height above $\tau$ allows us to define so-called \textit{$\veps$-trimmings} or \textit{$\veps$-simplifications} of $T$.  
\begin{definition}
The \textbf{$\veps$-simplified tree of $T$, $T^\veps$} or the \textbf{$\veps$-trimmed tree of $T$}, is the subtree of $T$ defined as
\be
T^\veps := \{\tau \in T \; \vert \;  h(\tau) \geq \veps\}
\ee
\end{definition}
An $\veps$-trimmed tree is always finite by virtue of the compactness of $T$. For a monotone decreasing sequence $(\veps_n)_{n\in\N}$ such that $\veps_n \to 0$, we have the following chain of inclusions
\be
T^{\veps_1} \xhookrightarrow{} T^{\veps_2} \xhookrightarrow{} T^{\veps_3} \xhookrightarrow{} \cdots \,.
\ee
Applying algorithm \ref{alg:barcodefromtree}, we get a set of maps on the persistence modules induced by these inclusions. More precisely, denoting $\Alg(T^{\veps_n})$ the output of the algorithm
\be
\Alg(T^{\veps_1}) \to \Alg(T^{\veps_2}) \to \Alg(T^{\veps_3}) \to \cdots \,.
\ee
where the morphisms are the maps induced at the level of the interval modules generating $\Alg(T^{\veps_n})$. Indeed, the interval modules $k[\al,\beta_n[$ of $\Alg(T^{\veps_n})$ satisfy that there is exactly one interval module of $\Alg(T^{\veps_m})$ ($m>n$) such that $[\al,\beta_n[ \,\subset [\al,\beta_m[$. A natural definition for infinite $T$ is thus
\be
\Alg(T) := \varinjlim \Alg(T^{\veps_n}) \,.
\ee 
\par
In categoric terms, the algorithm above in fact is a functor
\be
\Alg : \bf{Tree} \to \bf{PersMod}_k \;,
\ee
where $\bf{Tree}$ is the category of rooted $\R$-trees seen as metric spaces, whose morphisms are isometric embeddings (which are not required to be surjective) preserving the roots, and where $\bf{PersMod}_k$ is the category of q-tame persistence modules over a field $k$ (\cf Oudot's book for details on the category of persistence modules \cite{Oudot:Persistence}). The action of $\Alg$ on morphisms between two trees $\zeta : T \to T'$ is defined as follows. If both $T$ and $T'$ are finite, since $\zeta$ is an isometric embedding and it preserves the root, we can define $\Alg(\zeta)$ to be
\be
\Alg(\zeta) := \DS_i \id_{k[\zeta(\al_i),\zeta(\beta_i)[} \,,
\ee
where $k[\al_i,\beta_i[$ denotes the modules in the interval module decomposition of $\Alg(T)$ (which is finite, since $T$ is as well). If $T$ is infinite, we extend the above definition by taking successive $\veps_n$-simplifications of $T$ and taking the direct limit of the construction above. Note that this procedure is well-defined since $\veps_n$-simplifications only depend on the function $h$, which in turn can be taken to only depend on the distance to the root.
\subsubsection{Trees stemming from a function}
Let us now consider a tree $T_f$ stemming from a function $f$ and show that $\Alg(T_f) = H_0(X,f)$.
\begin{proposition}
\label{prop:handpathexistence}
Let $\tau$ and $\eta$ be elements of $T_f$ such that $f(\tau) < f(\eta)$ and let $x \in \pi^{-1}(\tau)$ and $y \in \pi^{-1}(\eta)$, then
\be
\exists \text{ path } \gamma : x \mapsto y  \text{ s.t. } \forall t, \;f(\gamma(t)) \geq f(\tau) \iff h(\tau) \geq f(\eta)-f(\tau) \text{ and } x, y \in X_{f(\tau)}^\tau \,.
\ee
\end{proposition}
\begin{proof}
Since there exists $\gamma$ connecting $x$ and $y$ and since $\gamma$ always stays above $f(\tau)$, we conclude naturally that $\Image(\gamma) \subset X_{f(\tau)}^\tau$, which implies that $h(\tau) \geq f(\eta) - f(\tau)$ by definition of $h(\tau)$. 
\par
The implication $(\Leftarrow)$ is clear since if $x,y \in X^\tau_{f(\tau)}$ and $X^\tau_{f(\tau)}$ is connected, by path connectedness of $X$ there exists a path between $x$ and $y$ which stays above $f(\tau)$. 
\end{proof}
This proposition suffices to prove the following theorem on the validity of algorithm \ref{alg:barcodefromtree}.
\begin{theorem}
Let $X$ be a compact, connected, locally path connected topological space and let $f: X\to \R$ be continuous. Then $\Alg(T_f) = H_0(X,f)$.
\label{thm:TfcalulatesH0}
\end{theorem}
\begin{remark}
This theorem is a slight improvement on the result of Curry in \cite[Theorem \S 2.13]{Curry_2018}. In the language of \cite{Curry_2018}, this constitutes a proof of the ``Elder rule'' with less assumptions of regularity. Indeed, in \cite{Curry_2018}, the assumption of a Morse set (or that $f$ is a Morse function) is necessary for the proof, whereas the functions hereby considered are merely required to be continuous. %{\color{red} We need to be precise here about Curry's assumption on regularity}
\end{remark}
\begin{remark}
By setting $X = T$ and $\ell = f$, theorem \ref{thm:TfcalulatesH0} states that $\Alg(T) =H_0(T,\ell)$.
\end{remark}
\begin{proof}
Suppose that $T_f$ is finite, then $\Alg(T_f)$ is a decomposable persistence module $\Alg(T_f):= \VV$. The fact that $\VV$ is pointwise isomorphic to $H_0(X,f)$ holds since $d_f$ correctly identifies the connected components of the superlevel sets. This guarantees the existence of a pointwise isomorphism since both spaces have the same (finite) dimension. 
\par
Let us now check that $\rk(\VV(r \to s)) = \rk(H_0(X_r \to X_s))$. The inclusion $X_r \xhookrightarrow{} X_s$ induces the following long exact sequence in homology
\begin{equation}
\begin{tikzcd}[row sep=small, column sep = small]
\cdots \arrow[r] & H_{1}(X_s) \arrow[r] \arrow[r] & H_{1}(X_s, X_r) \arrow[r] & H_0(X_r) \arrow[d]    &       &         \\
         &          &        & H_0(X_s) \arrow[r] & H_0(X_s,X_r) \arrow[r] & 0 
\end{tikzcd} \nonum
\end{equation}
Since this sequence is exact
\be
\rk(H_0(X_r \to X_s)) = \dim \ker(H_0(X_s) \to H_0(X_s,X_r)) \,.
\ee
For notational simplicity, let us denote $\phi:H_0(X_s) \to H_0(X_s,X_r)$. Note that $\phi[c] = [0]$ if and only if there is a path $\gamma$ between the representative $c \in X_s$ and an element $b \in X_r$ such that $\gamma$ stays within $X_s$. Without loss of generality, let us take $c$ such that $c \in \{f=s\}$. Finding such a path $\gamma$ is only possible if $c$ and $b$ lie in the same connected component of $X_r$. By proposition \ref{prop:handpathexistence}, this can happen if and only if $h([c]_{T_f}) \geq r-s$. It follows that
\be
\dim \ker \phi  = \#\{\tau \in T_f \, \vert \,  h(\tau) \geq r-s \, \text{ and }\, f(\tau)=s\} \,,
\ee
which concludes the proof for the finite case.
\par
If $T_f$ is infinite, we consider a sequence of $\veps_n$-trimmings of $T_f$ such that $\veps_n \xrightarrow[n \to \infty]{} 0$. For any $r>s$, there exists $n$ such that $r-s > \veps_n$. But $T_f^{\veps_n}$ is finite, so we are reduced to the previous case.
\end{proof}

\subsection{The inverse problem}
\label{sec:inverseproblem}
An interesting question is whether every (compact) tree stems from a function $f: X \to \R$. If the tree is a so-called \textit{merge tree} (in particular, we require that it be locally finite and 1-dimensional), a solution has been provided by Curry in \cite[\S 6]{Curry_2018}. We will now positively answer this question under the assumptions that $\updim T < \infty$ and that $X=[0,1]$ by constructing a function $f:[0,1] \to \R$, which constitute a wider class of trees than merge trees. The rest of this section will focus on proving the following theorem:
\begin{theorem}
Let $T$ be a compact $\R$-tree such that $\updim T < \infty$. Then, for any $\delta>0$ it is possible to construct a continuous function $f: [0,1] \to \R$ of finite $(\updim T + \delta)$-variation such that $T= T_f$. In particular, up to a reparametrization, $f$ can be taken to be $\frac{1}{\updim T + \delta}$-H\"older continuous. 
\label{thm:functionfromtree}
\end{theorem}
The idea is to once again use $\veps$-simplifications $T^\veps$ for which we can construct a function by taking the contour of the tree. Such a construction is referred to as the Dyck path in the terminology of \cite{Stanley_2009}.

\subsubsection{Finite trees}
We can regard a rooted discrete tree as being an operator with $N$ inputs, where $N$ is the number of leaves of the tree. There is a natural operation on the space of discrete trees which composes these operations by:
\begin{figure}[h!]
  \centering
    \includegraphics[width=0.5\textwidth]{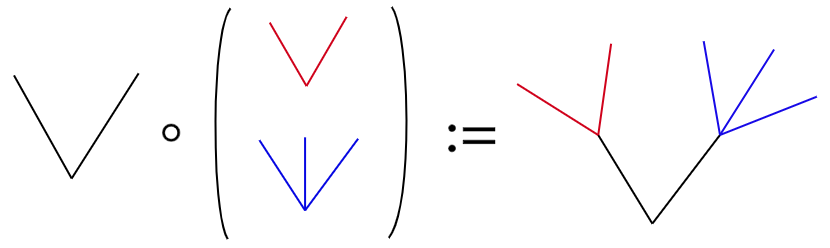}
\end{figure}
\par
These objects are called \textbf{operads} and originated in the study of iterated loop spaces \cite{May_1972, Boardman_1973,Boardman_1968}. Since then, these objects have been studied in different fields for a variety of purposes \cite{Loday_1996,Ginzburg_1994}. We will not give the explicit definition of an operad here, as a rigorous introduction is unnecessary for our purposes. However, we introduce this notion of composition of trees for notational simplicity.
\begin{figure}[h!]
  \centering
    \includegraphics[width=0.6\textwidth]{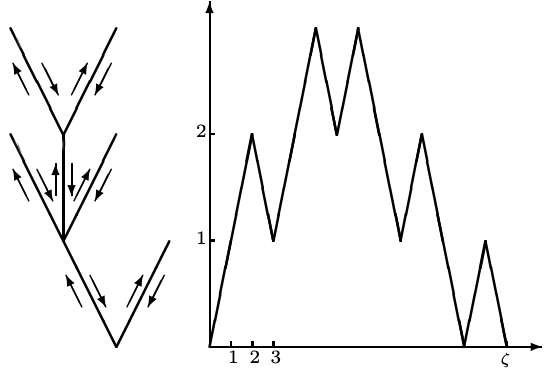}
  \caption{The Dyck path is the function $f$ which assigns the height (the distance from the root) of each vertex of the tree as we wrap around the tree following a clockwise contour around it. There is a map $\phi: T \mapsto [0,\zeta]$ where $[0,\zeta]$ is now marked at the points at which $f$ achieves its local maxima. The figure is taken from \cite{LeGall:Trees}.}
\label{fig:Dyckpath}
\end{figure}
\par
Given a discrete $\R$-tree $T$, if we have an embedding of $T$ in $\R^2$, or equivalently, a partial order on its vertices, we can assign to $T$ an interval $I$ of a certain length with $N$ marked points as well as a function $f_T : I \to \R$, where $N$ is the number of leaves of $T$. Using the terminology of \cite{Stanley_2009}, a way to do this is by considering the so-called \textbf{Dyck path} or \textbf{contour path} where the path around $T$ parametrized by arclength in $T$. The construction of the Dyck path has been carefully detailed in \cite{LeGall:Trees, Stanley_2009}, but it is better understood by looking at figure \ref{fig:Dyckpath}. By construction the equality: $T_{f_T} = T$ holds for any discrete $\R$-tree $T$. Here, equality is taken up to isometry.
\par
As per the description of figure \ref{fig:Dyckpath}, the construction of the Dyck path yields a map $\phi$ which to $T$ assigns an interval $\phi(T)$ with $N$ marked points. An example of the action of $\phi$ is illustrated in figure \ref{fig:actionofphi}.
\begin{figure}[h!]
  \centering
    \includegraphics[width=0.35\textwidth]{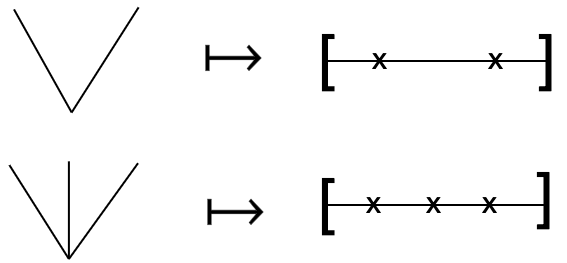}
  \caption{The action of $\phi$ on trees with two and three leaves respectively. The length of the intervals assigned is exactly the length of the contour around the trees and the marked points are the points at which $f_T$ achieves its maxima.}
\label{fig:actionofphi}
\end{figure}
\par
This operation $\phi$ is in fact a ``morphism'' with respect to a composition operation on the intervals, defined as follows. If we have an interval $I$ with $N$ marked points and $N$ intervals $J_k$ each with $M_j$ marked points, the result of the operation $I \circ (J_1, \cdots, J_N)$ is the insertion of the marked interval $J_k$ at the $k$th marked point of $I$. The length of $I \circ (J_1, \cdots, J_N)$ is
\be
\abs{I \circ (J_1, \cdots, J_N)} = \abs{I}+\sum_{k=1}^n \abs{J_k} \,,
\ee
where $\abs{\cdot}$ denotes the lengths of the intervals. The fact that $\phi$ is a ``morphism'' results from the definitions of compositions for trees and intervals. We can also define a variant of this morphism $\phi$, which we will call $\phi_\lambda$, which for any tree $T$ simply scales the (marked) interval $\phi(T)$ by a factor $\lambda$.\par
Given a tree $T$ the Dyck path $f_T: \phi(T) \to \R$ can be transformed into a function $f^\lambda_T: \phi_\lambda(T) \to \R$ by setting
\be
f^\lambda_T(x) := f_T(x/\lambda) \,.
\ee
This is a rescaling of the $x$-axis which means that $T_{f^\lambda_T}= T_{f_T} = T$ still holds. Once again, these equalities are taken up to isometry.
\begin{remark}
The definition of $f^\lambda_T$ is readily generalizable to forests. If $\mathcal{F}$ denotes a forest, then we define $f_{\mathcal{F}}^\lambda = \bigsqcup_{T \in \mathcal{F}} f_T^\lambda$.
\end{remark}
For discrete trees, there is an upper bound of the number of vertices of the tree given its number of leaves.
\begin{lemma}
Let $T$ be a rooted discrete tree, $N \geq 2$ be its number of leaves and $V$ be its number of vertices, then
\be
V \leq 2 N - 1 \,.
\ee
In particular, if the edges of $T$ all have length $1$, the contour of the tree can be done over an interval of length at most $4N-2$
\label{lemma:combinatoriallemma}
\end{lemma}
\begin{proof}
For binary trees, it is known that \cite{LeGall:Trees,Stanley_2009}
\be
V = 2N -1 \,.
\ee
Given a tree with $N$ leaves, we can obtain a binary tree with $N$ leaves by blowing up the vertices which are non-binary. The inequality of the lemma follows. On a binary tree, the Dyck path passes through almost every point in $T$ twice, so the length of the interval is exactly $4N-2$. Since binary trees are the extremal case, a bound for all trees with $N$ leaves follows.
\end{proof}
The results above show the result of theorem \ref{thm:functionfromtree} for finite trees, since their upper-box dimension is equal to 1. 
\subsubsection{Infinite trees}
The concatenation of trees can be defined for $\R$-trees too in the obvious way. Given an infinite number of compositions, we can define a limit tree by defining it to be the limit of the partial compositions in the Gromov-Hausdorff sense. Ideally, we would like to have an equality of the following type
\be
T = T^a \circ \overline{(T \setminus T^a)} \,,
\ee
where $T \setminus T^a$ now denotes the rooted forest corresponding to the set $T \setminus T^a$. This equality is desirable because by taking infinitely many compositions, we can eventually recover the original tree $T$, by composing successive $\veps_n$-simplifications with each other. However, this equality does not hold since $T^a$ might not have the right amount of leaves for this operation to be well-defined. Nonetheless, we can decide to count the vertices $T^a \cap \overline{(T \setminus T^a)}$ as leaves with multiplicity, so that the equality above holds. 

For an infinite compact tree with $\updim T < \infty$, the idea is to take some appropriate rapidly decreasing (monotonous) sequence $(\veps_n)_{n \in \N^*}$ such that the interval
\be
I = \phi_{\veps_1}(T^{\veps_1}) \circ \phi_{\veps_2}(T^{\veps_2} \setminus T^{\veps_1}) \circ \phi_{\veps_3}(T^{\veps_3} \setminus T^{\veps_2}) \circ \cdots
\ee
has finite length. On each $\phi_{\veps_k}(T^{\veps_k} \setminus T^{\veps_{k-1}})$ we can consider the Dyck path on the forest $T^{\veps_k} \setminus T^{\veps_{k-1}}$. Defining a correct superposition of these Dyck paths, we would be done (\cf figure \ref{fig:functionconstruction}).
\begin{figure}[h!]
  \centering
    \includegraphics[width=0.7\textwidth]{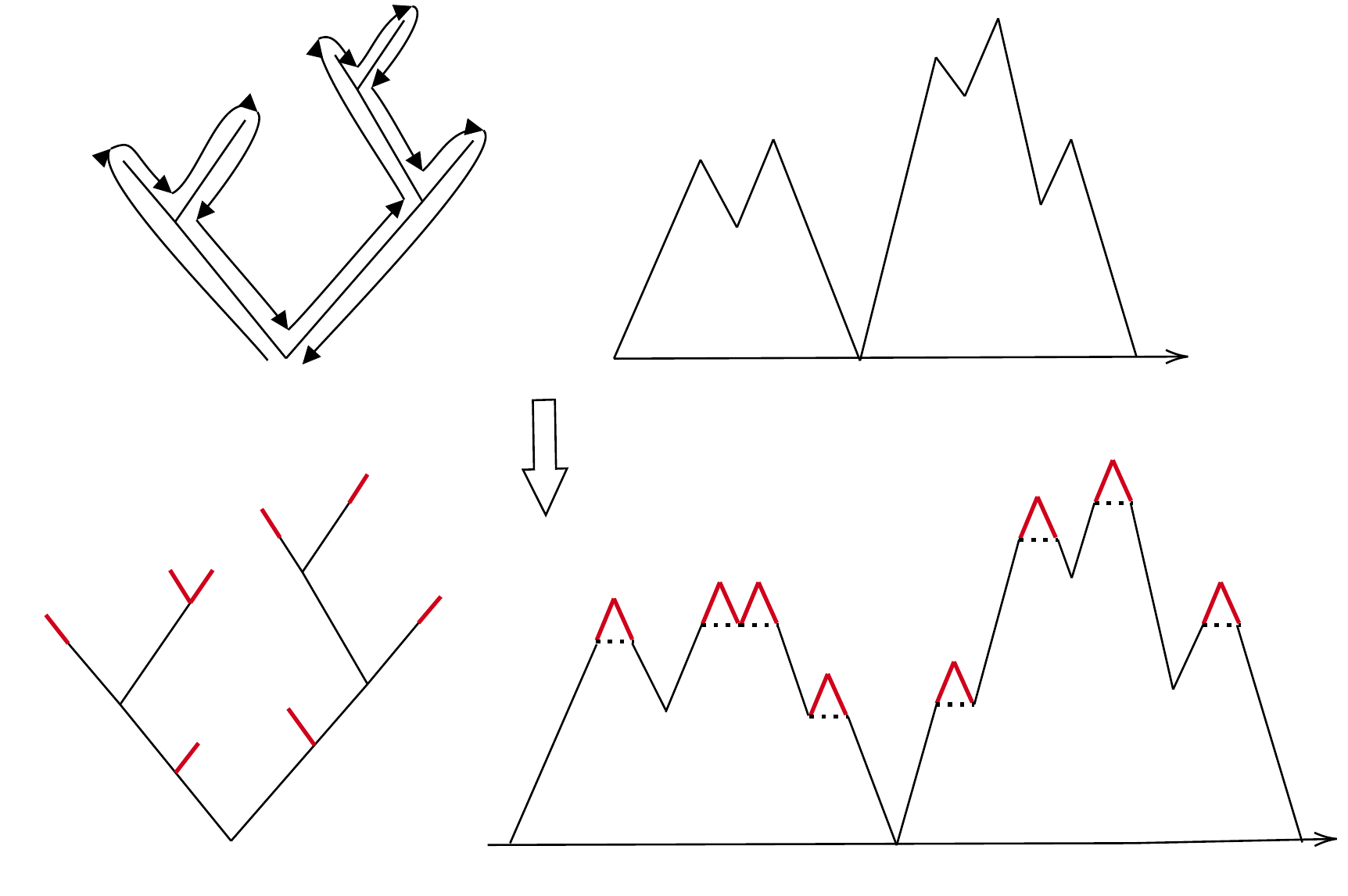}
  \caption{Starting from a tree $T^{a/2^k}$ (black) we construct the Dyck path around it in the first step. Then, we look at $T^{a/2^{k+1}}$ which leads to the addition of intervals (dotted), and a correction of the function at the $k$th step $f_k$ (which is the function depicted in black, extended linearly over the new intervals). We can further define a function by pasting the Dyck paths of the forest over the corresponding leaves, which leads to the function depicted in the second step (red and black).}
\label{fig:functionconstruction}
\end{figure}
For an infinite tree, it suffices to show that the sequence generated by the procedure of figure \ref{fig:functionconstruction} converges in the Gromov-Hausdorff sense to an interval of finite length $I$ and that $(f_i)_i$ converge in $L^\infty(I)$ to some function $f$.

\subsubsection*{Detailed construction of the approximants}
\begin{definition}
Let $I \subset \R_+$ be a marked interval with $n$ marked points, which we will denote $(i_k)_{\{1\leq k \leq n\}}$. Furthermore, let $(J_k)_{\{1\leq k \leq n\}}$ be a set of $n$ marked intervals of $\R_+$, each with $j_k$ marked points. Define $\sigma_I: I \to I \circ (J_1, \cdots J_n)$ by
\be
\sigma_I(x; J_1, \cdots, J_n) := \left[x +\!\!\!\!\!\!\!\! \sum_{i=1}^{\argmax_k \{i_k < x\}} \!\!\!\!\!\!\!\!\abs{J_i}\; \;\right] \;\; \in I \circ (J_1, \cdots, J_n) \,.
\ee
%This definition naturally extends to the whole interval $I$ and we can define the image $\sigma_I(I;J_1 \cdots J_n)$.% which is diagramatically represented in figure \ref{fig:sigmafunction}.
% \begin{figure}[h!]
%   \centering
%     \includegraphics[width=0.9\textwidth]{Figures/sigmafunction.png}
%   \caption{An illustration of the definition of the image of $I$ by the $\sigma$-function (in red).}
%   \label{fig:sigmafunction}
% \end{figure}
\end{definition}
\begin{remark}
Fixing $J_1, \cdots, J_n$, $\sigma_I$ is a bijective map onto its image, meaning every point $y \in \sigma_I(I;J_1, \cdots, J_n)$ admits a preimage in $I$, which we will denote by $\sigma^{-1}_I(y;J_1, \cdots, J_n)$.
\end{remark}
\begin{definition}
Let $f:I \to \R$ be a continuous function from an interval $I$ with $n$ marked points and let $(J_1, \cdots, J_n)$ be intervals with each with $j_i$ marked points as before. Abusing the notation, we define another function $\sigma(-;J_1, \cdots, J_n)$ which assigns a function on $I$ to a function on $\sigma_I(I;J_1, \cdots, J_n)$ via the following formula
\be
\sigma_I(f;J_1, \cdots, J_n)(x) := \begin{cases} f(\sigma^{-1}_I(x;J_1, \cdots, J_n)) & x \in \sigma_I(I;J_1,\cdots,J_n) \\ \text{Linearly extend elsewhere}\end{cases}
\ee
\end{definition}
\begin{remark}
By continuity of $f:I \to \R$, this linear extension on $I \circ(J_1, \cdots, J_n)$ is in fact constant everywhere outside $\sigma_I(I; J_1, \cdots, J_n)$ (this is the dotted region in figure \ref{fig:functionconstruction}). Note also that $\sigma_I(f;J_1, \cdots, J_n)$ is continuous.
\end{remark}

\begin{definition}
Given a tree $T_f$ associated to a continuous function $f$, we define:
\begin{itemize}
  \item The \textbf{projection onto the tree} as the mapping
  \begin{align}
  \pi : \; &X \to T_f = X/\{d_f=0\} \,; \\
    & x \mapsto [x]
  \end{align}
  \item Let $\tau \in T_f$, define the \textbf{left preimage of $\tau$, $\ola{\tau}$} and the \textbf{right preimage of $\tau$ by $\pi$, $\ora{\tau}$} as
  \begin{align}
  \overleftarrow{\tau} := \inf \pi^{-1}(\tau) \\
  \overrightarrow{\tau}:= \sup \pi^{-1}(\tau)\,.
  \end{align}
\end{itemize}
\label{def:projectionontotree}
\end{definition}

\begin{definition}
Let $T$ be a discrete rooted tree and $T' \subset T$ be a subtree sharing roots with $T$ and suppose that we have chosen some embedding of $T$. Suppose there is a function $f: I \to \R$ on a certain interval $I$ such that $T_f = T'$. Then, the marking of $I$ induced by $T$ is the marking induced by marking the preimage $\pi_f^{-1}(T' \cap \overline{(T\setminus T')})$ chosen in the following way:
\begin{itemize}
  \item If $\tau \in T' \cap \overline{(T\setminus T')}$ admits a single preimage, choose this preimage;
  \item Else, if the connected component of $\tau$ in $\overline{T \setminus T'}$ is smaller (with respect to the partial order on the tree induced by the embedding of $T$) than every vertex strictly greater than $\tau \in T'$, choose $\ola\tau$. Otherwise, choose $\ora\tau$. In simpler terms, we choose $\ora\tau$ or $\ola\tau$ depending on whether the subtree of $\overline{T\setminus T'}$ containing $\tau$ branches to the right or to the left respectively of $T'$, with the convention that we say that it branches to the left if lies at the top of a leaf of $T'$ (\cf figure \ref{fig:treebranchleftright}). 
\end{itemize} 
We will denote this marking operation by $\mu(I;T',T,f)$.
\end{definition}

\begin{figure}[h!]
  \centering
    \includegraphics[width=0.5\textwidth]{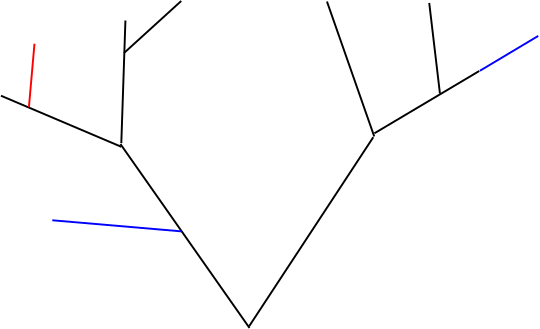}
  \caption{A tree $T$ embedded in $\R^2$ with a subtree $T'$ in black, the subtrees highlighted in red branch to the right and those in blue to the left.}
\label{fig:treebranchleftright}
\end{figure}

We can also define analogous maps to $\sigma_I$, but this time on the intervals $J_k$ as follows.
\begin{definition}
Let $I \subset \R_+$ be a marked interval with $n$ marked points, which we will denote $(i_k)_{\{1\leq k \leq n\}}$. Furthermore, let $(J_k)_{\{1\leq k \leq n\}}$ be a set of $n$ marked intervals of $\R_+$, each with $j_k$ marked points. Define $\eta_I^{J_k} : J_k \to I \circ (J_1, \cdots, J_n)$ by
\be
\eta_I^{J_k}(x; J_1, \cdots, J_n) := x + i_k + \sum_{j=1}^{k-1} \abs{J_j}\,.
\ee
These maps define a map $\eta_I = \bigsqcup_k \eta_I^{J_k}$ on $\bigsqcup_k J_k$ and $\eta_I$ also induces a map on the functions $f: \bigsqcup_k J_k \to \R$, defined analogously to $\sigma_I$, which we shall also denote $\eta_I$. 
\end{definition}
With this notation, the construction is made in accordance to algorithm \ref{alg:fnsints}. A depiction of the mechanism of algorithm \ref{alg:fnsints} can be found in figure \ref{fig:functionconstruction}.
\begin{algorithm}[h!]
\SetAlgoLined
\textbf{Output:} A set of unions of intervals $(I_i)_{i \in \{1, \cdots, n\}}$ and a set of functions on $I_n$, $(f_i :I_n \to \R)_{i \in \{1, \cdots, n\}}$ \\
\textbf{Input:} An infinite tree $T$ and $a>0$.

 $I_1 \gets \phi(T^a)$ \;
 $f_1 \gets f_{T^a}$ \;
 $I \gets I_1$ \;
 $i \gets 1$ \;
 \While{$i\leq n$}{
  $I_{i+1} := I_i \circ \phi_{\lambda^i}(\overline{T^{a/2^{i+1}} \setminus T^{a/2^{i}}})$ \;
  $f \gets \eta_{I_{i+1}}(f^{\lambda^{i+1}}_{T^{a/2^{i+1}} \setminus T^{a/2^{i}}} ; I_1, \cdots I_i)$ \;
  $I_i \gets \mu(I_i;T^{a/2^{i-1}},T^{a/2^i},f_i)$ \;
  \For{j=1; $j \leq i$}{
    $I_j \gets \sigma(I_j;\phi_{\lambda^i}(\overline{T^{a/2^{i+1}} \setminus T^{a/2^{i}}}))$ \;
    $f_j \gets \sigma(f_j;\phi_{\lambda^i}(\overline{T^{a/2^{i+1}} \setminus T^{a/2^{i}}}))$ \;
    $j \gets j+1$ \;
  }
  $f_{i+1} := f_i +  f$ \;
  $i \gets i+1$ \;
 }
 \Return $(I_i)_{i \in \{1, \cdots, n\}}$, $(f_i)_{i \in \{1, \cdots, n\}}$. 
 \caption{Construction of approximants}
\label{alg:fnsints}
\end{algorithm}
\begin{figure}[h!]
  \centering
    \includegraphics[width=0.7\textwidth]{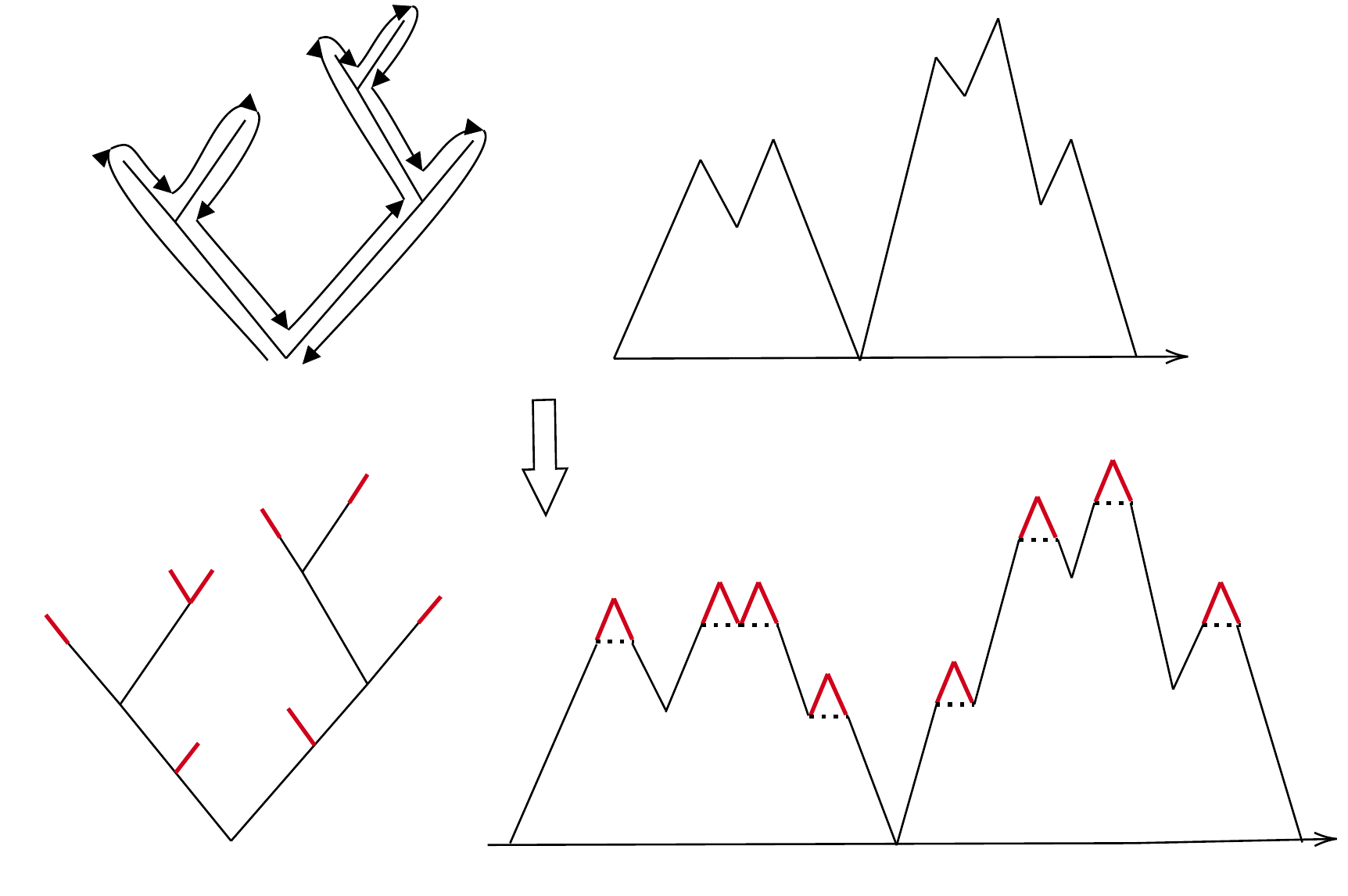}
  \caption{Starting from a tree $T^{a/2^k}$ (black) we construct the Dyck path around it in the first step. Then, we look at $T^{a/2^{k+1}}$ which leads to the addition of intervals (dotted), and a correction of the function at the $k$th step $f_k$ (which is the function depicted in black, extended linearly over the new intervals). We can further define a function by pasting the Dyck paths of the forest over the corresponding leaves, which leads to the function depicted in the second step (red and black).}
\label{fig:functionconstruction}
\end{figure}
For an infinite tree, it suffices to show that the sequence generated by this algorithm converges in the Gromov-Hausdorff sense to an interval of finite length $I$ and that $(f_i)_i$ converge in $L^\infty(I)$ to some function $f$.

\subsubsection*{End of the proof}
To get the desired convergence we must show the two following lemmata.
\begin{lemma}
\label{lemma:goal1}
If $T$ is a compact $\R$-tree of finite upper-box dimension, there exist $a$ and $\lambda$ such that $I$ defined by the construction above has finite length.
\end{lemma}  
We need to show the convergence of the corresponding functions $(f_n)_n$. This can be done by proving that the sequence is Cauchy. 
\begin{lemma}
\label{lemma:fnisCauchy}
Given the definition of functions $f_n$ above, then the sequence $(f_n)_{n \in \N^*}$ is Cauchy in $C^0(I)$, we have
\be
\norm{f_n - f_m}_{C^0} \leq a2^{-(n \wedge m)}
\ee 
for any $n$ and $m \in \N^*$. 
\end{lemma}
By completeness of $C^0$, the sequence $(f_n)_{n \in \N^*}$ uniformly converges to a continuous function $f$. By virtue of stability theorem for trees (theorem \ref{thm:diagstreesqiso}) it follows that $T$ is isometric to $T_f$. Using Picard's theorem (theorem \ref{thm:PicardThm})
\be
\mathcal{V}(f) = \updim T_f = \updim T
\ee
which concludes the proof of theorem \ref{thm:functionfromtree}. 

\begin{proof}[Proof of lemma \ref{lemma:goal1}]
Recall that, according to the proof of theorem \ref{thm:Lfandupboxdim}, the following equality holds for any tree $T$
\be
\limsup_{\veps \to 0} \frac{\log N^\veps}{\log(1/\veps)} \vee 1 \leq \updim T := \al \,.
\ee
Unpacking the definition of the lim sup, for any $\delta >0$ there is a $a>0$ such that for all $\veps < a$, we have that
\be
N^\veps < \veps^{-\al-\delta} \,.
\ee
Let us fix such a $\delta$ and pick $a$ small enough so that the condition above holds. For any $n \in \N^*$, the partial composition of intervals has length
\be
\abs{I_n} = \abs{\phi(T^a)} + \sum_{k=1}^n \abs{\phi_{\lambda^k}(T^{a/2^{k}} \setminus T^{a/2^{k-1}})} \,.
\ee
However, we can bound $\abs{\phi_{\lambda^k}(T^{a/2^{k}} \setminus T^{a/2^{k-1}})}$ by
\begin{align}
\abs{\phi_{\lambda^k}(T^{a/2^{k}} \setminus T^{a/2^{k-1}})} &= \lambda^k \abs{\phi(T^{a/2^{k}} \setminus T^{a/2^{k-1}})} \nonum \\
&\leq \lambda^k \left(\frac{a}{2^k}\right)(4N^{a/2^k})\,,
\end{align}
since on $T^{a/2^{k}} \setminus T^{a/2^{k-1}}$ the distances between the vertices of each tree are at most $a/2^k$ and there are at most $4N^{a/2^k}$ such edges by virtue of lemma \ref{lemma:combinatoriallemma}. Thus,
\be
\abs{\phi_{\lambda^k}(T^{a/2^{k}} \setminus T^{a/2^{k-1}})} < 4\lambda^k \left(\frac{a}{2^k}\right)^{1-\al-\delta} = 4 a^{1-\al-\delta} \left(2^{\al+\delta-1}\lambda\right)^k \,.
\ee
Setting $\lambda < 2^{1-\al-\delta}$ $I_n$ converges to some interval of finite length $I$, since the partial sums $\abs{I_n}$ converge.
\end{proof}

\begin{proof}[Proof of lemma \ref{lemma:fnisCauchy}]
Suppose that $n<m$. It is sufficient to show that on $I_m$ the equality holds, since in all further iterations of the algorithm, the functions $f_n$ and $f_m$ are locally constant over the intervals introduced. By definition of $f_n$, $f_n$ and $f_m$ agree on $I_n$. Outside of this set, $f_n$ is constant and the difference in the $\Linfty$-norm depends only on what happens above $T^{a/2^n}$, thus we can write
\be
\norm{f_n - f_m}_{\Linfty} \leq \norm{f_{T^{a/2^m} \setminus T^{a/2^n}}}_{\Linfty}
\ee
by definition of $f_n$. However, the Dyck path on $T^{a/2^m} \setminus T^{a/2^n}$ can at most reach a height of $a(2^{-n}-2^{-m})< a 2^{-n}$, which finishes the proof.
\end{proof}

\section{Regularity, persistence index and metric properties of trees}
Throughout this section $X$ will be a compact, connected and locally path-connected metric space. On general topological spaces, it is important to specify which homological theory we are using to compute the homology of $X$. For nice enough spaces, this choice has little to no importance, as most homological theories coincide. However, for abstract metric spaces this is no longer necessarily the case. For our purposes, we will always consider the homology of the space $X$ to be its \v{C}ech homology.
\textit{A priori}, this might pose some problems, as \v{C}ech homology does not always satisfy the axioms of a proper homological theory in the sense of Eilenberg-Steenrod. For this to be the case, a sufficient condition is to consider $X$ to be compact and the homology to be taken over a field. These are not the only conditions for which \v{C}ech homology gives rise to a proper homological theory, as in general the exactness axiom might fail, but suffices for our purposes. For more on these technical details, we encourage the reader to consult Eilenberg's book \cite[Chapter 7]{Eilenberg_1952}.
\begin{remark}
If we wish to consider higher degrees of homology over even more general topological spaces where the exactness axiom does indeed fail for the \v{C}ech homology, there are multiple options. We could either consider more elaborate homology theories such as singular homology or strong homology (which fixes the issue with the exactness axiom of \v{C}ech homology), or we could rewrite this paper in cohomological terms and use \v{C}ech cohomology, for which this problem doesn't present itself. 
\end{remark}
With this technicality out of the way, let us now define the main objects which will concern us for the rest of this paper. 
\begin{definition}
Let $X$ be a compact, connected, locally path connected topological space and consider $f: X \to \R$ be a continuous function. The $k$th \textbf{$\Pers_p$-functional of $f$} is
\be
\Pers_p(H_k(X,f)) := \left(\sum_{b \in H_k(X,f)} \ell(b \cap [\inf(f),\sup(f)])^p \right)^{1/p} \,,
\ee
where $\ell(b)$ denotes the length of the bar $b$ and $H_k(X,f)$ denotes the $H_k$-barcode (or diagram) stemming from the superlevel filtration. Abusing the notation, we will denote $\Pers_p(f) := \Pers_p(H_0(X,f))$. If we further assume that there exists $n$ such that for all $m > n$, $H_m(X) =0$, we define the \textbf{total $\Pers_p$ functional of $f$} as
\be
\text{TPers}_p(f) := \sum_{k=0}^n \Pers_p(H_k(X,f)) \,.
\ee
\end{definition}

\begin{definition}
\label{def:PersIndex}
Let $f: X \to \R$ be a continuous function. The \textbf{$k$th-persistence index of $f$} is defined as
\be
\Lag_k(f) := \inf\{p \geq 1 \, \vert \, \Pers_p(H_k(X,f)) < \infty \} \,.
\ee
We will sometimes write $\Lag(f) := \Lag_0(f)$. Provided that higher degrees of homology identically vanish, we may also talk about the \textbf{total persistence index of $f$}, defined as
\be
\Lag_{Tot}(f) := \inf\{p \geq 1 \, \vert \, \sum_k \Pers_p(H_k(X,f))  <\infty\} \,.
\ee
\end{definition}

\subsection{1D case: a connection with the $p$-variation}

\begin{definition}
Let $f: [0,1] \to \R$ be a continuous function. The \textbf{true $p$-variation of $f$} is defined as
\be
\norm{f}_{p-\text{var}} := \left[\sup_D \sum_{t_k \in D} \abs{f(t_{k}) - f(t_{k-1})}^p\right]^{1/p} \,,
\ee
where the supremum is taken over all finite partitions $D$ of the interval $[0,1]$.
\end{definition} 
\begin{remark}
We talk about \textit{true} $p$-variation to make the distinction with the notion of variation typically considered in probabilistic contexts (more precisely, stochastic calculus), where instead of the supremum over all partitions, we have a probable limit as the mesh of the partition considered tends to zero. 
\end{remark}

\begin{proposition}[Picard, \S 3 \cite{Picard:Trees}]
\label{thm:PicardThm}
Let $f: [0,1] \to \R$ be a continuous function, then $\norm{f}_{p-\text{var}}$ is finite as soon as $\Pers_p(f)$ is finite. In fact, for any $p$
\be
\norm{f}_{\pvar}^p  \leq 2\Pers_p^p (f)^p \,.
\ee
Furthermore, if $\norm{f}_{(p-\delta)\text{-var}}$ is finite for some $\delta >0$, $\Pers_p(f)$ is also finite.
\label{thm:lpsameaspvar}
\end{proposition}
In fact, Picard showed that on the interval $[0,1]$, the persistence index of $f$ is linked to the regularity of $f$.
\begin{theorem}[Picard, \S 3 \cite{Picard:Trees}]
Let $f: [0,1] \to \R$ be a continuous function and denote
\be
\mathcal{V}(f) := \inf\{p \, \vert \, \norm{f}_{p\text{-var}} <\infty\} \,.
\ee
Then,
\be
\mathcal{V}(f) = \Lag(f)  = \limsup_{\veps \to 0} \frac{\log(\lambda(T_f^\veps)/\veps)}{\log(1/\veps)} +1 = \limsup_{\veps \to 0} \frac{\log N^\veps}{\log(1/\veps)} \vee 1= \updim T_f
\ee
where $a \vee b := \max\{a,b\}$, $N^\veps$ is the number of leaves of the $\veps$-trimmed tree $T_f^\veps$, $\lambda(T_f^\veps)$ denotes the length of $T_f^\veps$ and $\updim$ denotes the upper-box dimension.
\end{theorem}
\begin{remark}
More generally, we can define $\lambda$ as the unique atomless Borel measure on $T_f$ characterized by the fact that the measure of a geodesic is given by the length of the geodesic \cite{Picard:Trees}.
\end{remark}

\subsection{More general spaces}
\subsubsection{Connected, locally path-connected, compact topological spaces}
\begin{theorem}
\label{thm:Lfandupboxdim}
Let $X$ be a connected, locally path-connected, compact topological space and let $f: X \to \R$ be a continuous function. With the same notation as above and supposing that $\updim T_f$ is finite, the following chain of equalities holds
\be
\Lag(f)  = \limsup_{\veps \to 0} \frac{\log N^\veps}{\log(1/\veps)} \vee 1 =\limsup_{\veps \to 0} \frac{\log(\lambda(T_f^\veps)/\veps)}{\log(1/\veps)} +1 = \updim T_f \,.
\ee
Furthermore,
\be
\liminf_{\veps \to 0} \frac{\log N^\veps}{\log(1/\veps)} \vee 1  \leq \downdim T_f \leq \liminf_{\veps \to 0} \frac{\log(\lambda(T_f^\veps)/\veps)}{\log(1/\veps)} +1\,,
\ee
where $\downdim$ is the lower-box dimension. For $\downdim T_f >1$, these inequalities turn into equalities if either:
\be
\limsup_{\veps \to 0} \frac{N^{2\veps}}{N^\veps} <1 \quad \text{or} \quad \limsup_{\veps \to 0} \frac{\lambda(T_f^{2\veps})}{\lambda(T_f^\veps)} <1 \,.
\ee
\end{theorem}
\begin{remark}
The study of $N^\veps$ is in fact completely equivalent to the study of $\Pers_p^p(f)$. Indeed,
\be
\Pers_p^p(f) = p\int_0^\infty \veps^{p-1} N^\veps \; d\veps \,,
\ee
which is finite as soon as $p > \Lag(f)$. This is nothing other than the Mellin transform of $N^\veps$. By the Mellin inversion theorem, for any $c >\Lag(f)$, we have
\be
N^\veps = \frac{1}{2\pi i} \int_{c-i\infty}^{c+i \infty} \Pers_p^p(f) \,\veps^{-p} \;\frac{dp}{p} \;.
\ee
\end{remark}
\begin{proof}[Proof of theorem \ref{thm:Lfandupboxdim}]
By the procedure detailed in section \ref{sec:inverseproblem}, since $\updim T_f$ is finite we can construct a function $\hat{f} : [0,1] \to \R$ such that $T_f$ and $T_{\hat{f}}$ are isometric. Applying Picard's theorem to $T_{\hat{f}}$ and noting that $\Lag(f)$ depends only on the $T_f$, we have that
\be
\Lag(f)  = \limsup_{\veps \to 0} \frac{\log (\lambda(T_f^\veps)/\veps)}{\log(1/\veps)} +1 = \limsup_{\veps \to 0} \frac{\log N^\veps}{\log(1/\veps)} \vee 1= \updim T_f \,.
\ee 
Let us now show the inequalities for the $\liminf$. Since
\be
\lambda(T_f^\veps) = \int_\veps^\infty N^a \;da \,,
\ee
the following inequality holds
\be
\liminf_{\veps \to 0} \frac{\log N^\veps}{\log(1/\veps)} \vee 1 \leq \liminf_{\veps \to 0} \frac{\log (\lambda(T_f^\veps)/\veps)}{\log(1/\veps)} +1\,.
\ee
Additionally,
\be
N^\veps \leq \NN(\veps/2)
\ee
where $\NN(\veps)$ denotes the minimal number of balls of radius $\veps$ necessary to cover $T_f$. This inequality holds as above each leaf of $T_f^\veps$, at least one ball of radius $\frac{\veps}{2}$ is necessary to cover this section of the tree. It follows that
\be
\label{ineq:liminfNvepsLambdaveps}
\liminf_{\veps \to 0} \frac{\log N^\veps}{\log(1/\veps)} \vee 1 \leq \downdim T_f \,.
\ee
We can bound this minimal number of balls $\NN(\veps)$ by the following
\be
\NN(\veps) \leq N^{\veps/2} + \frac{\lambda(T_f^{\veps/2})}{\veps/2} \leq 2 \; N^{\veps/2} \vee \ceil{\frac{\lambda(T_f^{\veps/2})}{\veps/2}} \,,
\ee 
which holds since, at most $N^\veps$ balls are needed to cover $T_f \setminus T_f^\veps$. To cover $T_f^\veps$, at most: $\ceil{\lambda(T_f^{\veps/2})/(\veps/2)}$ balls are needed, so the inequality above follows by further majorizing the terms. This implies that
\be
\downdim T_f  \leq \left[\liminf_{\veps \to 0} \frac{\log N^\veps}{\log(1/\veps)} \vee 1 \right]\vee \left[ \liminf_{\veps \to 0} \frac{\log (\lambda(T_f^\veps)/\veps)}{\log(1/\veps)} +1 \right] \,,
\ee
but by inequality \ref{ineq:liminfNvepsLambdaveps} this means that
\be
\downdim T_f \leq \liminf_{\veps \to 0} \frac{\log (\lambda(T_f^\veps)/\veps)}{\log(1/\veps)} +1 \,.
\ee
Finally,
\begin{align}
\frac{\lambda(T_f^\veps)- \lambda(T_f^{2\veps})}{\veps} &= \frac{1}{\veps} \left[\int_\veps^\infty N^a \; da - \int_{2\veps}^\infty N^a \; da\right] \nonum\\
&= \frac{1}{\veps} \int_\veps^{2\veps} N^a \; da \leq N^\veps  \,,
\end{align}
since $N^\veps$ is monotone decreasing. This reasoning also gives a lower bound
\be
N^{2\veps} \leq \frac{\lambda(T_f^\veps) - \lambda(T_f^{2\veps})}{\veps} \leq N^\veps \,,
\ee
which entails that
\be
\liminf_{\veps \to 0} \frac{\log N^\veps}{\log(1/\veps)}  = \liminf_{\veps \to 0} \frac{\log \left[\frac{\lambda(T_f^\veps) - \lambda(T_f^{2\veps})}{\veps}\right]}{\log(1/\veps)} \,.
\ee
Suppose that this limit is larger than $1$. Rearranging, we get
\be
\frac{\veps N^{2\veps}}{\lambda(T_f^\veps)}\leq 1-\frac{\lambda(T_f^{2\veps})}{\lambda(T_f^\veps)} \leq \frac{\veps N^\veps}{\lambda(T_f^\veps)} \,,
\ee
from which it follows that if any of these quantities admits a $\liminf$ which is stricly greater than zero, we have 
\be
\liminf_{\veps \to 0} \frac{\log N^\veps}{\log(1/\veps)} = \liminf_{\veps \to 0} \frac{\log \lambda(T_f^\veps)}{\log(1/\veps)} +1 \,.
\ee
Noticing another equivalent condition for the validity of this equality is whether
\be
\limsup_{\veps \to 0} \frac{N^{2\veps}}{N^\veps} <1 \,,
\ee
finishes the proof.
\end{proof}

\begin{remark}
If $\updim = \downdim$, all the limits of the above theorem are well-defined, yielding exact asymptotics for $\lambda(T_f^\veps)$ and $N^\veps$. This is in particular the case if $\updim = \dim_H$, where $\dim_H$ denotes the Hausdorff dimension. 
\end{remark}

The functional $\lambda(T^\veps_f)$ is what some authors \cite{Polterovitch:Persistence, Polterovitch:LaplaceEigenfunctions} refer to as the Banach indicatrix and its asymptotics have a topological interpretation as described in the statement of the theorem. It is interesting to note that the study of the upper-box dimension is natural in the tree approach. Additionally, $\updim$ has also been used in the context of persistent homology by Schweinhart \cite{Schweinhart_2019}, Schweinhart and MacPherson \cite{MacPherson_2012} and by Adams \textit{et al.} \cite{Adams_2020} in a probabilistic setting. 

\subsubsection{LLC metric spaces}
It is possible to further extend Picard's theorem by some rudimentary considerations and by imposing the so-called locally linearly connected condition on $X$. Let us briefly recall the definition of this condition.
\begin{definition}
A \textbf{locally linearly connected (LLC) metric space $(X,d)$}, is a connected metric space such that for all $r>0$ and for all $z \in X$, for all $x,y \in B(z,r)$, there exists an arc connecting $x$ and $y$ such that the diameter of this arc is linear in $d(x,y)$. 
\label{def:LLCspace}
\end{definition}
With this extra assumption, we can prove the following lemmata.
\begin{lemma}[Regularity-dimension]
\label{lemma:RegularityDimension}
Let $X$ be a compact LLC metric space. Keeping the same notations as in theorem \ref{thm:Lfandupboxdim}, the following inequality holds
\be
\Lag(f) = \updim T_f \leq \Ham(f) \updim X \,,
\ee
where:
\be
\Ham(f) := \inf \left\{ \frac{1}{\al} \; \Big\vert \, \exists \lambda \in \Homeo(X)\,, \;\norm{f \circ \lambda}_{C^\al} < \infty\right\}
\ee
\end{lemma}
The proof of this lemma relies itself on two lemmata, which are interesting in and of themselves.
\begin{lemma}
\label{lemma:dimYleqdimX}
Let $X$ and $Y$ be two metric spaces such that there is a surjective map $\pi: X \to Y$ such that $\pi \in C^\al(X,Y)$, then
\be
\updim Y  \leq \frac{1}{\al}\, \updim X  \quad \text{and} \quad \downdim Y \leq \frac{1}{\al} \downdim X\,.
\ee 
\end{lemma}
\begin{lemma}
\label{lemma:fCalpiCal}
Let $X$ be a compact locally linearly connected (LLC) metric space (\cf definition \ref{def:LLCspace}) and let $f: X \to \R$ be a continuous function, then
\be
f \in C^\al(X,\R) \Longrightarrow  \pi_f \in C^\al(X,T_f) \,.
\ee
\end{lemma}
Let us show that lemmata \ref{lemma:dimYleqdimX} and \ref{lemma:fCalpiCal} imply lemma \ref{lemma:RegularityDimension}.
\begin{proof}[Proof of lemma \ref{lemma:RegularityDimension}]
If, up to precomposition, $f \notin C^\al(X,\R)$ for any $\al$, there is nothing to show, since the statement is vacuous. Otherwise, since $T_f$ is preserved by precomposition by a homeomorphism, we may suppose without loss of generality that $f \in C^\al(X,\R)$. The projection onto the tree of $f$, $\pi_f : X \to T_f$ is in $C^\al(X,T_f)$ according to lemma \ref{lemma:fCalpiCal}. It follows from lemma \ref{lemma:dimYleqdimX} that
\be
\updim T_f \leq \frac{1}{\al} \, \updim X \,.
\ee
The statement of the theorem follows by taking the infimum over $\frac{1}{\al}$.
\end{proof}
All that remains to show is the two remaining lemmata.
\begin{proof}[Proof of lemma \ref{lemma:dimYleqdimX}]
Since $\pi : X \to Y$ is surjective and $C^\al(X,Y)$, for any $x \in X$
\be
\pi\!\left(B_X\!\left(x, \left(\frac{\veps}{K}\right)^{1/\al}\right)\right) \subset B_Y(\pi(x),\veps) 
\ee
for some constant $K$. It follows that the minimal number of balls needed to cover $X$, $\NN_X$ dominates the minimal number of balls needed to cover $Y$, $\NN_Y$. More precisely
\be
\NN_Y(\veps) \leq \NN_X\!\left(\left(\frac{\veps}{K}\right)^{1/\al}\right) \iff \al \;\frac{\NN_Y(\veps)}{\log(1/\veps) + \log(K)} \leq \frac{\NN_X\!\left(\left(\frac{\veps}{K}\right)^{1/\al}\right)}{\log\!\left(\left(\frac{K}{\veps}\right)^{1/\al}\right)} \,.\nonum
\ee
The statement of the lemma follows.
\end{proof}
\begin{proof}[Proof of lemma \ref{lemma:fCalpiCal}]
Suppose that $f: X \to \R$ is in $C^\al(X, \R)$ with H\"older constant $\Lambda$ and let $x,y \in X$. Without loss of generality, suppose that $f(x) < f(y)$. Since $T_f$ is a geodesic space, the distance $d_f(\pi_f(x), \pi_f(y))$ is the length of the geodesic arc in $T_f$ linking $\pi_f(x)$ and $\pi_f(y)$. By compactness of this geodesic path, there is a point $\tau \in T_f$ where $f$ achieves its minimum, thus
\be
d_f(\pi_f(x),\pi_f(y)) = f(x)- f(\tau)  + f(y) - f(\tau) \;.
\ee
This minimum $f(\tau)$ has the particularity that
\be
f(\tau) = \sup_{\gamma: x \mapsto y} \inf_{t \in [0,1]} f \circ \gamma \;,
\ee
where the supremum is taken over all paths on $X$ linking $x$ and $y$. From the LLC condition, we know that there is a path $\eta: x \mapsto y$ whose diameter is controlled by $d_X(x,y)$ and $z \in X$ achieving the minimum of $f$ over $\eta$. In particular,
\be
f(\tau) \geq \inf_{t \in [0,1]} f \circ \eta =: f(z) \,.
\ee 
Since $f$ is $\al$-H\"older on $X$,
\be
f(x) - f(\tau) \leq  f(x)-f(z) \leq \Lambda \; d(x,z)^\al \leq \Lambda \;\text{diam}(\eta)^\al \leq  C\Lambda\; d(x,y)^\al
\ee
for some constant $C$ determined by the LLC condition and we have an analogous inequality for $f(y)-f(\tau)$. Putting everything together we have that:
\be
d_f(\pi_f(x),\pi_f(y)) \leq 2C \Lambda\; d_X(x,y)^\al\,,
\ee
which finishes the proof.
\end{proof}
Lemma \ref{lemma:RegularityDimension} is sharp, since Brownian sample paths almost surely saturate this inequality. However, there is no hope to prove equality for every $f$. Indeed, for any $f \in C^1(\mathbb{T}^2,\R)$ having a finite amount of bars, $T_f$ is a finite tree and has upper-box dimension $1$, but
\be
\updim T_f = 1 < 2 = \Ham(f) \updim \mathbb{T}^2 \;.
\ee
Nonetheless, it is possible to show that lemma \ref{lemma:RegularityDimension} holds generically. This is a consequence of a generalization of work never published by Weinberger and Baryshnikov. We extend their result to homogenous enough spaces in the following sense.
\begin{definition}
A metric space $(X,d)$ is said to \textbf{admit a homogeneous set} (for a certain property) if there exists an open set $U \subset X$ where for every ball $B(x,r) \subset U$, the property of the ball is the same as the property of the space $X$. 
\end{definition}
\begin{remark}
In the previous definition, one can for instance take any notion of dimension, entropy, \textit{etc}.
\end{remark}
The following proposition will be useful in simplifying the assumptions of the theorem.
\begin{proposition}
\label{prop:PackingVsCover}
Let $(X,d)$ be a compact metric space and $N_P(\veps)$ denote the cardinality of the maximal packing of $X$ by balls of radius $\veps$. Then,
\be
\mathcal{N}_X(2\veps) \leq N_P(\veps) \leq \mathcal{N}_X(\veps) \,
\ee
and in particular, 
\be
\downdim(X) = \liminf_{\veps \to 0} \frac{\log(N_P(\veps))}{\log(1/\veps)} \quad \text{and} \quad \updim(X) = \limsup_{\veps \to 0} \frac{\log(N_P(\veps))}{\log(1/\veps)}\,.
\ee
\end{proposition}
\begin{proof}
Let $M_\veps$ be a maximal packing of $X$ by balls of radius $\veps$. For every $x \in X \setminus (\cup_{V \in M_\veps} V)$ there exists $U \in M_\veps$ such that $d(x,U) \leq \veps$, otherwise, $B(x,\veps) \cup M_\veps$ would also be a packing of $X$ with cardinality strictly greater than $\abs{M_\veps}$. It follows that the balls of radius $2\veps$ of centers that of the maximal packing of radius $\veps$ is a covering of $X$, proving the first inequality. 

For the second inequality, we reason by contradiction. Suppose there is a maximal packing $P_\veps$ and a minimal covering $C_\veps$ such that $\abs{P_\veps} \geq \abs{C_\veps}+1$. Then, since $C_\veps$ covers $X$, by the pigeonhole principle there are at least two centers of balls of $P_\veps$ inside a ball of $C_\veps$. But the triangle inequality implies that the balls around these two centers of radius $\veps$ have non-empty intersection (as the center of the ball of $C_\veps$ in which they are contained is in the intersection), thereby contradicting that $P_\veps$ is a packing, showing the result.
\end{proof}
\begin{theorem}
\label{thm:genericityofHomDim}
Let $X$ be a compact LLC space admitting a set of homogeneous lower-box dimension, then for any $0<\al\leq 1$
\be
\sup_{f \in C^\al(X,\R)} \al \Lag(f) = \updim(X)\,.
\ee 
Moreover, the supremum is attained generically in the sense of Baire, \textit{i.e.} the set over which $\al \Lag(f) < \updim(X)$  is meagre in $C^\al(X,\R)$.
\end{theorem}
% \begin{remark}
% If the space $X$ is triangulable, the same statement holds for the \textit{total $p$-persistence} of the function $f$, \textit{i.e.} the sum over all the $\Pers_p$-functionals of the $k$th degree persistent homology diagrams of $f$. 
% \end{remark}
Once again, we split the proof along key lemmata.
\begin{lemma}
\label{lemma:PerspvepsContinuity}
Let $X$ be a compact LLC space, then the functional $\Pers_{p,\veps}^p: C^\al_\Lambda(X,\R) \to \R_+$ defined by 
\be
f \mapsto \sum_{\substack{b \in \bcode(f) \\ \ell(b) \geq \veps}} \ell(b)^p
\ee 
is continuous.
\end{lemma}
\begin{proof}
We start by noting that the total number of bars of length $\geq \veps$ that a function $f \in C^\al_\Lambda(X,\R)$ can have is uniformly bounded above by virtue of the proof of lemma \ref{lemma:dimYleqdimX} by a constant $C_{X,\al,\veps}$. By lemma \ref{lemma:fCalpiCal}, we know $\pi_f : X \to T_f$ is $\al$-H\"older, with H\"older constant $K$ depending only on $\Lambda$ and $X$. This fact, combined with the inequality $N^\veps_f \leq \mathcal{N}_{T_f}(\veps/2)$ entails that for any $f$,
\be
N^\veps_f \leq \mathcal{N}_{T_f}(\veps/2) \leq \mathcal{N}_X\left(\left(\frac{\veps}{2K}\right)^{1/\al}\right)=: C_{X,\al,\veps} \,.
\ee
It follows that for any $f, g \in C^\al_\Lambda(X,\R)$, by choosing to sum along the $d_\infty$-matching, we have 
\begin{align*}
\abs{\Pers_{p,\veps}^p(f)-\Pers_{p,\veps}^p(g)} &\leq \sum_{\substack{b_f \in \bcode(f) \,,\; b_g \in \bcode(g)  \\ \ell(b_f), \ell(b_g) \geq \veps}}\abs{\ell(b_f)^p - \ell(b_g)^p} \\
&\leq \sum p \underbrace{\abs{\ell(b_f)- \ell(b_g)}}_{\leq \norm{f-g}_\infty \text{ by stability}}\,\max\{\ell(b_f)^{p-1},\ell(b_g)^{p-1}\}  \\
&\leq p  \norm{f-g}_\infty \underbrace{\sum \max\{\ell(b_f)^{p-1},\ell(b_g)^{p-1}\}}_{\leq C_{X,\al,\veps} \Lambda^{p-1} \,\text{diam}(X)^{\al(p-1)} \text{ by global $\al$-H\"olderness}} \\
&\leq C_{X,\al,\veps}\, \Lambda^{p-1} \text{diam}(X)^{\al(p-1)} \,  p \, \norm{f-g}_\infty \,. 
\end{align*}
\end{proof}
\begin{lemma}
\label{lemma:DenseSetOfInfinitePersp}
Let $X$ be a compact, LLC, admitting a set of homogeneous lower-box dimension. Then, for all $p <\updim(X)$ and $M\geq0$, the set of functions 
\be
\{f \in C^\al(X,\R) \, \vert \, \Pers_p^p(f) > M\}
\ee
is dense in $C^\al(X,\R)$. 
\end{lemma}
\begin{proof}
Without loss of generality, suppose that the uniform set is a ball of radius $1$ inside $X$, denoted $B \subset X$ and construct a function $h$ of persistence $>M$ on this ball. Noting $d = \updim(X)$, by proposition \ref{prop:PackingVsCover} and the definition of the upper-box dimension, for some subsequence of $(\veps_n)_n$ decreasing to $0$, we have
\be
\tilde{C}\veps_n^{-(d-\delta)} \leq N_P(\veps) \leq  C\veps_n^{-(d+\delta)}
\ee
for some constants $C$ and $\tilde{C}$. Note $E_\veps$ the centers of the balls of a maximal packing of radius $\veps$ and define $h_n: B \to \R$ as
\be
h_n(x) := d^\al(x,E_{\veps_n})
\ee
The $\Pers_p$-functional of these functions can be bounded below by
\be
\Pers_p^p(h_n) \geq N_P(\veps_n) \veps_n^p \geq \tilde{C} \veps_n^{p\al-d+\delta} 
\ee
for all $\delta>0$. Since $\al p<d$, this quantity can be made as large as we want and in particular $>M$ by picking a large enough $n$. By the assumptions of the theorem, it is possible to choose the original ball of the construction to have as small a radius as we wish. Note we may perturb any function $f \in C^\al(X,\R)$ by a function close to it which is locally constant on a small enough ball and on this ball, add $h_n$ for $n$ large enough. Since the ball of the construction can be chosen as small as we want, any neighborhood of $f$ contains a function satisfying the condition of the lemma.  
\end{proof}

\begin{proof}[Proof of theorem \ref{thm:genericityofHomDim}]
We are interested in showing that for $p<\updim(X)$, the set 
\be
\mathcal{S}(p) := \{f \in C^\al(X,\R) \, \vert \, \Pers_p^p(f) < \infty \}
\ee
is meager in $C^\al(X,\R)$. Let us start by noticing that
\be
\mathcal{S}(p) = \bigcup_{\Lambda \geq 0} \bigcup_{M \geq 0} \mathcal{S}(p,\Lambda,M) \,,
\ee
where the union is taken over an increasing diverging sequences of $\Lambda$ and $M$ and 
\be
\mathcal{S}(p,\Lambda,M) := \{f \in C^\al_\Lambda(X,\R) \, \vert \, \Pers_p^p(f) \leq M \} \,.
\ee
Furthermore,
\be
\mathcal{S}(p,\Lambda,M) = \bigcap_{k\geq 1} \{f \in C^\al_\Lambda(X,\R) \, \vert \, \Pers_{p,\frac{1}{k}}^p(f) \leq M  \} \,.
\ee
By lemma \ref{lemma:PerspvepsContinuity}, $\Pers_{p,\frac{1}{k}}$ is continuous, thereby guaranteeing that these sets are closed in $C^\al(X,\R)$, and therefore so is their intersection. It remains to show that the $\mathcal{S}(p,\Lambda,M)$ are nowhere dense, but this amounts to finding a dense set of functions for which 
\be
\Pers_{p,\frac{1}{k}}^p(f) \leq M
\ee
is violated for infinitely many $k$. It suffices to find a dense set of functions for which the total $\Pers_p^p(f) >M$ (for $p < \updim(X)$), but the existence of such a dense family is given by lemma \ref{lemma:DenseSetOfInfinitePersp}, showing the result.
\end{proof}
\begin{remark}
The space defined by
\be
E_p= \{f \in C^0(X,\R) \, \vert \, \Lag(f) \leq p\}
\ee
is \textbf{not} a linear space. 
\end{remark}

\subsubsection{Doubling spaces with small convex balls}
One could ask whether the results of genericity of theorem \ref{thm:genericityofHomDim} hold in every degree of homology for $f$ within some class of regularity. This question has been considered in \cite{LipschitzStableLpPers} and more recently in \cite{Skraba_2020} with different degrees of generality. The following theorem is a slight generalization of the two cited results.

\begin{theorem}
\label{thm:genericityTotalHomology}
Let $X$ be a compact, connected geodesic doubling space whose small enough balls are geodesically convex. Denote $d = \updim(X)$, $k \in \N$ and let $f \in C^\al(X,\R)$, then $\Lag_k(f) \leq \frac{d}{\al}$. %Moreover, the set of $f$ with $\Lag_k(f) < \frac{d}{\al}$ is meagre in $C^\al(X,\R)$.  
\end{theorem}
\begin{remark}
The doubling assumption is satisfied for Riemannian manifolds whose Ricci curvature is bounded below, by the Bishop-Gromov inequality. By considering Gromov-Hausdorff limits of Riemannian manifolds with Ricci curvature bounded below, we obtain spaces satisfying the doubling property. Spaces included in this class include, but are not limited to, Riemannian manifolds with conic singularities. In general, it is also possible to obtain less well-behaved spaces. For more on poorly behaved examples, we refer the reader to the works of Xavier Menguy \cite{Menguy_2001,Menguy_Thesis} and to even more recent and poorly behaved examples, such as those described in \cite{Jiang_2021}.
\end{remark}
The proof relies on the two following well-known lemmata. 
\begin{lemma}[Nerve lemma, Lemma 4.11 \cite{Oudot:Persistence}]
\label{lemma:NerveLemma}
Let $X$ be a paracompact space, and let $\mathcal{U}$ be an open cover of $X$ such that the $(k+1)$-fold intersections of elements of $\mathcal{U}$ are either empty or contractible for all $k \in \N$. Then, there is a homotopy equivalence between the nerve of $\mathcal{U}$ and $X$.
\end{lemma}
\begin{lemma}
\label{lemma:SmallCvxBalls}
Let $(X,d)$ be a geodesic metric space whose balls of radius $\leq \veps$ are geodesically convex. Then, minimal coverings of $X$ by balls of radius $\leq \veps$ are such that the $(k+1)$-fold intersections of elements of $\mathcal{U}$ are either empty or contractible for all $k \in \N$.
\end{lemma}
\begin{proof}[Proof of theorem \ref{thm:genericityTotalHomology}]
The proof is an immediate consequence of the proof of theorem \ref{thm:genericityofHomDim}, where we only need to modify the proof of lemma \ref{lemma:PerspvepsContinuity}. For this, it is sufficient to bound the number of bars in the persistence diagram of the $k$th degree in homology of length $\geq \veps$, $N_k^\veps$. On a given a minimal covering $\mathcal{U}$ of $X$ by balls of radius $\left(\frac{\veps}{4 \norm{f}_{C^\al}}\right)^{1/\al}$, $f$ varies by at most $\frac{\veps}{2}$ inside each ball. Given any $r \in \R$, construct the set $\mathcal{U}_r$ consisting in the union of all balls of $\mathcal{U}$ which intersect $X_r$. From this, we get a chain of inclusions
\be
X_r \xhookrightarrow{} \mathcal{U}_r \xhookrightarrow{} X_{r-\veps} \,,
\ee
which induces a chain of maps at the homology level. In particular, by functoriality of $H_*$,
\begin{align*}
\rk(H_*(X_r \to X_{r-\veps})) &\leq \rk(H_*(X_r \to \mathcal{U}_r)) \vee \rk(H_*(\mathcal{U}_r \to X_{r-\veps})) \\
&\leq \dim(H_*(\mathcal{U}_r))\,,
\end{align*}
but for small enough $\veps$, the covering's homology is the homology of its nerve by lemmas \ref{lemma:NerveLemma} and \ref{lemma:SmallCvxBalls} so this dimension is bounded above by the cardinality of the nerve. It follows that $N_k^\veps$ is bounded above by the cardinality of the $k$-skeleton of the nerve of such a minimal covering. Since the space is doubling, this yields the upper bound 
\be
N^\veps_k \leq (M^{k+1}-M^{k})\;\mathcal{N}_X\!\left(\left(\frac{\veps}{4 \norm{f}_{C^\al}}\right)^{1/\al}\right) \,,
\ee
where $M$ is the doubling constant of $X$. The rest of the proof follows from previous arguments without extra difficulty.
\end{proof}
\begin{remark}
\label{rmk:NonDoubling}
If the space is not supposed to be doubling, the only bound we have on $N_k^\veps$ is given by $\mathcal{N}_X^{k+1}$, which yields an analogous statement for $\Lag_k(f) \leq \frac{d(k+1)}{\al}$.
\end{remark}
%The proof is an immediate consequence of \cite{LipschitzStableLpPers}, where for the $\al$-H\"older case we replace the distance $d$ on $X$ by $d^\al$ and the genericity result is an immediate consequence of theorem \ref{thm:genericityofHomDim}. 
Under a supplementary assumption, we can show that the inequality obtained in theorem \ref{thm:genericityTotalHomology} is in fact generically an equality. As before, the genericity result relies on the existence of functions whose $\Pers_p$ functional for $p<\frac{d}{\al}$ is arbitrarily large. For this we rely on the following theorem of Divol and Polonik.
\begin{theorem}[Divol and Polonik, \cite{Polonik_2019}]
\label{thm:DivolPolonik}
Let $\mu$ be a bounded probability measure on $[0,1]^d$ and let $\mathbf{X}_n := (X_1, \cdots, X_n)$ be a vector of i.i.d. samples of $\mu$, then for $0<p<d$ and $0\leq k < d$ then almost surely,
\be
\lim_{n \to \infty} n^{-1+\frac{p}{d}}\Pers_p^p(H_k([0,1]^d,d(-,\mathbf{X}_n)) \to \Pers_p^p(\nu^\mu_p) \,,
\ee
for some non-degenerate Radon measure depending on $p$ and the probability measure $\mu$, $\nu^\mu_p$ on $\XX$. 
\end{theorem}
With this result we are now ready to prove the following theorem.
\begin{theorem}
\label{thm:SchweinhartAns}
Let $X$ be a compact Riemannian manifold of dimension $d$. Then, generically in the sense of Baire in $C^\al(X,\R)$, for any $0 \leq k <d$, $\Lag_k(f)= \frac{d}{\al}$.
\end{theorem}
\begin{proof}
The proof of genericity is essentially the same as that of theorem \ref{thm:genericityofHomDim}, with the exception that we now need to modify lemma \ref{lemma:DenseSetOfInfinitePersp}. The existence of a function $h$ with arbitrarily large $\Pers_p$-functional for $p<\frac{d}{\al}$ on any small ball is given by Divol and Polonik's construction by tweaking the filtration in their proofs from being the distance $d$ to $d^\al$. As before, this entails the genericity result for the set of functions of $C^\al$ satisfying $\Lag_k(f) \geq \frac{d}{\al}$. Compact Riemannian manifolds have strictly positive convexity radii and Ricci curvature bounded below, and so satisfy the hypotheses of theorem \ref{thm:genericityTotalHomology}, applying the theorem yields the desired equality.  
\end{proof}

\subsection{A partial answer to a question by Schweinhart}
\label{sec:Schweinhart}
In \cite{Schweinhart_2019}, Schweinhart introduces a notion of persistent homology dimension of a metric space $X$, defined as follows. 
\begin{definition}[Schweinhart's definition of $\dim_{\PH}^k$, \cite{Schweinhart_2019}]
Let $X$ be a bounded subset of a metric space. The $k$th homological dimension of $X$ is
\be
\dim_{\PH}^k(X):= \sup_{\xx} \inf_p \; \{\Pers_p(H_k(X,d(-,\xx))) <\infty\} \,,
\ee
where the supremum is taken over all finite sets of points $\xx$ of $X$.
\end{definition}
Given our previous results, we suggest the following modification to this definition, for reasons which will become apparent later. 
\begin{definition}[$k$th homological dimension of $X$]
\label{def:HomDim}
Let $X$ be a bounded subset of a metric space. The $k$th homological dimension of $X$ is defined as
\be
\dim_{\PH}^k(X):= \sup_{f \in \Lip_1(X)} \Lag_k(f) \,,
\ee
where $\Lip_1(X)$ denotes the set of Lipschitz functions with Lipschitz constant $\leq 1$.
\end{definition}
Theorem \ref{thm:genericityTotalHomology} already allows us to partially answer Schweinhart's Question 5 \cite{Schweinhart_2019}. However, this is not a complete answer, because one should make sure that there are Lipschitz functions on $X$ on the class of metric spaces as those of those of theorem \ref{thm:genericityTotalHomology} such that the inequality $\Lag_k(f) \leq d$ is saturated, or saturated to within $\delta$ for all $\delta >0$. Without the assumption that $X$ is doubling, an interesting question is whether the bound found is optimal: the proof of the theorem suggests that if such metric spaces exist, they cannot be of ``bounded geometry'' and are relatively pathological.

As we saw in theorem \ref{thm:SchweinhartAns}, this bound is saturated for any integer $0 \leq k < \updim(X)$ under the assumption that $X$ is a compact manifold. Thereby entailing
\be
\dim_{\PH}^k(X) = \updim(X)
\ee
for such $X$. Here, the notions of homological dimension of Schweinhart and our own coincide exactly, as the genericity result is proven via distance functions to point clouds. This thus establishes sufficient conditions for this equality to hold, albeit not necessary ones.

\section{Distance notions and stability properties of trees and diagrams}

\subsection{Some elements of optimal transport}
\subsubsection{Defining optimal partial transport}
Let us follow the exposition by Divol and Lacombe \cite{Divol_2019}, and quickly introduce optimal \textit{partial} transport, which extends optimal transport to measures of \textit{a priori} different masses (which may be potentially infinite), for a detailed account of the theory, we refer the reader to the cited article, but also to the works of different authors \cite{Figalli_2009,Chizat_2015,Kondratyev_2016}. Divol and Lacombe build on the work of Figalli \cite{FIGALLI2010107} and extend Wasserstein distances to Radon measures supported on open proper subsets $\mathcal{X}$ of $\R^n$, whose boundary is denoted by $\partial \mathcal{X}$ (and $\overline{\mathcal{X}} := \mathcal{X} \sqcup \partial \mathcal{X}$). The general idea is that we should look at $\del \XX$ as a reservoir of infinite mass, capable of accomodating for any disparity in the mass of the measures considered. In this way, if two Radon measures $\mu$ and $\nu$ have different mass, we can form still define a transport map from one measure to the other by sending the mass surplus to the boundary $\del \XX$. Symbolically,

\begin{definition}{\cite[Problem 1.1]{FIGALLI2010107}} \label{def:OT_for_Radon_measures}
  Let $p \in [1, +\infty)$. Let $\mu, \nu$ be two Radon measures supported on $\mathcal{X}$ satisfying
  \[ \int_\mathcal{X} d(x,\partial \mathcal{X})^p \; d \mu(x) < +\infty, \quad \int_\mathcal{X} d(x,\partial \mathcal{X})^p \;d \nu(x) < +\infty.\]
  The set of \textbf{admissible transport plans} $\Gamma(\mu,\nu)$ is defined as the set of Radon measures $\pi$ on $\overline{\mathcal{X}} \times \overline{\mathcal{X}}$ satisfying
\[ \pi(A\times \overline{\mathcal{X}}) = \mu(A) \quad \text{ and } \quad \pi(\overline{\mathcal{X}}  \times B) = \nu(B). \]
   for all Borel sets $A,B \subset \mathcal{X}$. Furthermore, the cost of $\pi \in \Gamma(\mu,\nu)$ is defined as
\begin{equation}
C_p(\pi) := \int_{\overline{\mathcal{X}} \times \overline{\mathcal{X}}} d(x,y)^p \;d \pi(x,y).
\label{eq:costDefinition}
\end{equation}
The \textbf{optimal transport distance $d_p(\mu,\nu)$} is defined as
\be
d_p(\mu,\nu) := \left( \inf_{\pi \in \Gamma(\mu,\nu)} C_p(\pi) \right)^{1/p}.
\label{eq:optimalCostPbm}
\ee
Plans $\pi \in \Gamma(\mu,\nu)$ realizing the infimum in equation \ref{eq:optimalCostPbm} are called \textbf{optimal}.
\end{definition}

\begin{definition}
The space of Radon measures on $\XX$ will be denoted $\mathcal{D}(\XX)$ (or simply $\DD$ if $\XX$ is clear from context). We also introduce the following spaces
\be
\DD_p := \left\{\mu \in \mathcal{D} \, \Bigg\vert \, \int_{\XX} d^p(x,\del \XX)\; d\mu(x) <\infty \right\} \,.
\ee
We further define $\DD_\infty$ as the space of Radon measures with compact support.
\end{definition}
% \begin{definition}{\cite[Definition 2.2]{Divol_2019}}
% \label{def:induced_transport_map}
% Let $\mu, \nu \in \mathcal{M}$. Consider $g : \overline{\mathcal{X}} \to \overline{\mathcal{X}}$ a measurable function satisfying for all Borel set $B \subset \mathcal{X}$
% \be
%  \mu(g^{-1}(B) \cap \mathcal{X}) + \nu(B \cap g(\partial \mathcal{X})) = \nu(B). 
% \ee
% Define for all Borel sets $A, B \subset \overline{\mathcal{X}}$, 
% \be
%  \pi(A \times B) = \mu(g^{-1}(B) \cap \mathcal{X} \cap A) + \nu(\mathcal{X} \cap B \cap g(A \cap \partial \mathcal{X})). 
% \ee
% $\pi$ is called the transport plan \textit{induced} by the \textit{transport map} $g$.
% \end{definition}
\begin{remark}
\label{rmk:limitptoinfty}
A proof by Th\'eo Lacombe shows that for optimal partial transport distances $d_p$ also satisfy $d_p \xrightarrow{p\to \infty}d_\infty$. Indeed, for any $\pi \in \Gamma(\mu,\nu)$ 
\be
C_p(\pi)  \xrightarrow{p \to \infty} C_\infty(\pi)
\ee
The space $\Gamma(\mu,\nu)$ is sequentially compact \cite[Proposition 3.2]{Divol_2019}, so up to extraction of a subsequence, $(\pi_p)_p$ admits a limit $\pi_\infty$. Finally, if $\pi^*$ is an optimal transport for the cost function $C_\infty$, then 
\be
C_\infty(\pi^*) = \lim_{p \to \infty} C_p(\pi^*) \geq \lim_{p \to \infty} C_p(\pi_p) = C_\infty(\pi_\infty) \,,
\ee
so $\pi_\infty$ also achieves $\inf_\pi C_\infty(\pi)$, showing the desired result.
\end{remark}

When considering optimal partial transport, there may be complications with respect to the conventional theory of optimal transport, because the measures may have infinite mass. This poses some problems, among others because of the unavalaibility of Jensen's inequality, which may render certain results of the classical theory false, or require alternative proofs. Luckily, most classical results we will need can be adapted to this more general setting. 

\subsubsection{Some results on optimal transport distances}
To distinguish the theory of optimal transport from that of optimal \textit{partial} transport, let us introduce the following notation. 
\begin{notation}
Let $(X,\delta)$ be a Polish metric space. Denote $\mathcal{P}(X)$ (or simply $\mathcal{P}$ is $X$ is clear from context) the set of \textbf{probability measures on $X$} and define
\be
\mathcal{P}_p(X) := \left\{ \mu \in \mathcal{P} \, \Bigg\vert \, \int_X \delta^p(x,x_0) \; d\mu(x) < \infty \right\} 
\ee
for some $x_0 \in X$ (note that this definition does not depend on $x_0$). Once again, we may omit $X$ if it is clear from context. For any two measures $\mu, \nu \in \mathcal{P}(X)$, slightly abusing then notation, we may define the \textbf{space of transport maps $\Gamma(\mu,\nu)$} to be the space of probability measures on $X^2$ having marginals $\mu$ and $\nu$. We equip the space $\mathcal{P}_p(X)$ with a Wasserstein distance, defined as
\be
W_{p,\delta}(\mu,\nu) := \inf_{\pi \in \Gamma(\mu,\nu)} \norm{\delta}_{L^p(\pi)} \,.
\ee
For the rest of this paper, the distance indicated by $W$ will always reserved to classical Wasserstein distances between \textit{probability} measures, whereas the distance denoted $d_p$ will always refer to the notion of Wasserstein distances between general Radon measures, previously described in the context of optimal \textit{partial} transport.

Many statements are valid whether we are in the optimal transport or the optimal \textit{partial} transport setting. For this reason, we introduce the following generic notation along with the following dictionary to transpose statements to one setting or another.
\begin{table}[h!]
\centering
\begin{tabular}{||c c c ||} 
 \hline
 Generic notation & Optimal transport & Optimal partial transport \\ [3pt]
 \hline
 $(Y,d)$ & $(X,\delta)$ & $(\XX,d)$ \\ [5pt]
 $\del Y$ & $x_0 \in X$ & $\del \XX$ \\ [5pt]
 $\OT_p$ & $W_{p,\delta}$ & $d_p$ \\ [5pt]
 $\Mel(Y)$ & $\mathcal{P}(X)$ & $\DD(\XX)$ \\ [5pt]
 $\Mel_p(Y)$ & $\mathcal{P}_p(X)$ & $\DD_p(\XX)$ \\ [5pt]
 \hline
\end{tabular}
\caption{Dictionary between optimal and optimal partial transport.}
\label{table:OTvsOPT}
\end{table}
\end{notation}

\begin{proposition}
For any $1\leq p <\infty$, $\OT_p^p$ is convex, in the sense that for every $\mu_1, \mu_2, \nu \in \Mel_p$ and $t \in [0,1]$,
\be
\OT_{p}^p(t\mu_1 + (1-t)\mu_2,\nu) \leq t \OT_{p}^p(\mu_1,\nu) + (1-t) \OT_{p}^p(\mu_2,\nu) \,.
\ee
Moreover, if $\nu_1, \nu_2 \in \mathcal{M}_p$,
\be
\OT_{p}^p(t\mu_1 + (1-t)\mu_2,t\nu_1 +(1-t)\nu_2)  \leq t \OT_p^p(\mu_1,\nu_1) + (1-t) \OT_p^p(\mu_2,\nu_2)\,.
\ee
\end{proposition}
\begin{proof}
For every $\pi_i \in \Gamma(\mu_i,\nu)$, $t\pi_1 + (1-t)\pi_2 \in \Gamma(t\mu_1+(1-t)\mu_2,\nu)$, so
\be
\OT_p^p(t\mu_1+(1-t)\mu_2,\nu) \leq t \int_{Y^2} d(x,y)^p \;d\pi_1(x,y) + (1-t)\int_{Y^2} d(x,y)^p \;d\pi_2(x,y) \,,
\ee
which yields the result by taking the infimum over $\pi_1$ and $\pi_2$ on the right-hand side. The second convexity result is obtained by an analogous proof.
\end{proof}
\begin{remark}
Convexity does not hold for $p=\infty$. By taking the $\frac{1}{p}$-th power of both sides and letting $p\to \infty$ in the inequality above, all that we may conclude is that
\be
\OT_\infty(t\mu_1 +(1-t)\mu_2,\nu) \leq \max\{\OT_\infty(\mu_1,\nu),\OT_\infty(\mu_2,\nu)\} \,.
\ee
\end{remark}
\begin{theorem}[$\OT_p$ for $p=\infty$, \cite{Givens_1984}]
\label{thm:Winftydelta}
The distance obtained on $\mathcal{M}_\infty(Y)$ from $\OT_p$ by taking $p\to \infty$ is well-defined and coincides with the distance defined by
\be
\OT_{\infty}(\mu,\nu) = \inf_{\pi \in \Gamma(\mu,\nu)} \norm{d}_{L^\infty(\pi)} \,.
\ee
Furthermore, we have the following characterization of $\OT_{\infty}$
\be
\OT_{\infty}(\mu,\nu) = \inf\left\{r>0 \; \vert \; \forall U \subset Y \text{ open, } \mu(U) \leq \nu(U^r) \text{ and } \nu(U) \leq \mu(U^r)\right\}\,,
\ee
where $U^r$ denotes an open tubular neighborhood of radius $r$ around $U$. 
\end{theorem}

\begin{remark}
The topology of $\OT_{\infty}$ is finer than that of weak convergence.
\end{remark}

\begin{proposition}
\label{prop:WassersteinHolder}
Let $f : (Y,\delta) \to (Y',\delta')$ be an $\al$-H\"older map with H\"older constant $\Lambda$ and let $\mu, \nu \in \mathcal{P}(Y)$ then 
\be
W_{p,\delta'}(f_{\sharp}\mu,f_\sharp \nu) \leq \Lambda\, W_{p\al, \delta}^\al(\mu,\nu) \,.
\ee
\end{proposition}
\begin{proof}
The inequality is an immediate consequence of the H\"older continuity of $f$.
\end{proof}

\subsubsection{Persistence measures}
Coming back to persistence theory, recall that it is possible to see persistence diagrams as measures on 
\be
\XX := \{(x,y) \in \R^2 \, \vert \, y>x\} \,.
\ee 
Henceforth, $\XX$ will always refer to this half space. Seen as measures, persistence diagrams are nothing other than a sum of Dirac measures. Closing this space with respect to the topology of vague convergence, we retrieve the set of Radon measures on $\XX$. 
\begin{definition}
The set of \textbf{persistence measures} $\DD$ is the set of Radon measures (of potentially infinite mass) on $\XX:= \{(x,y) \in \R^2 \, \vert \, y>x\}$.
\end{definition}
Equipping $\XX$ with the $\ell^\infty$-distance on $\R^2$ defined by 
\be
d((p,q),(r,s)) = \max\{\abs{p-r}, \abs{q-s}\} \,,
\ee
optimal partial transport distances $d_p$ between persistence measures become definable. The repercussions of this have been explored by Divol and Lacombe in \cite{Divol_2019}. 

The extension from the space of persistence diagrams to the space of persistence measures has three main advantages. First, that, as shown in \cite{Divol_2019}, it is possible to use the machinery of optimal transport to address problems in persistence theory. Second, that $\DD$ is a linear space, which renders taking means and combinations of diagrams possible and easy. Finally, that it is well-adapted to the stochastic setting, because of the linearity property and Tonnelli's theorem: two key properties which we will exploit repeatedly.
\begin{remark}
The notion of average as defined in the linear space of persistence measures in general exits the space of persistence diagrams. This can for instance be seen by considering a sequence of measures which vaguely tend to a measure which is absolutely continuous with respect to the Lebesgue measure on $\XX$. In this case, it is impossible to reconstruct a function whose diagram agrees with the desired measure. This is obvious in the 1D case where it is impossible to construct any tree from such a persistence measure, and so by extension, to construct any function.
Nonetheless, this notion of average has the advantage of encoding the averages of all linear functionals of the diagrams (one can in fact see this as a definition of this notion of average by adopting a dual point of view). 
Some authors have considered alternative notions of central tendencies adapted to metric spaces (and in particular the space of diagrams), such as Fréchet means defined on the spaces of diagrams (\cf for instance the work of Turner \textit{et al.} \cite{Turner_2014}). While this notion stays in the space where persistent diagrams are defined, it depends on the distance chosen on $\DD$ and moreover also on the exponent chosen for the cost function in the definition of Fréchet means. 
\end{remark}

\subsection{Stability of Wasserstein $p$-distances on diagrams}
With respect to optimal transport distances, we have some ``stability theorems'' the most classical of which is
\begin{theorem}[Bottleneck stability with respect to $\Linfty$, Corollary 3.6 \cite{Oudot:Persistence}]
Let $f,g: X \to \R$ be two continuous functions, then 
\be
d_\infty(\Dgm(f),\Dgm(g)) \leq \norm{f-g}_{\infty}
\ee
where $\Dgm(f)$ and $\Dgm(g)$ denote the diagrams of $f$ and $g$ respectively.
\end{theorem}

\begin{theorem}[Wasserstein $p$ stability]
\label{thm:WassersteinpStabilityLLC}
Let $X$ be a compact LLC metric space of $\updim(X) = d$ and consider $f, g \in C^\al_\Lambda(X,\R)$. Then, for all $p > q>\frac{d}{\al}$,  
\be
d_p^p(H_0(X,f),H_0(X,g)) \leq C_{X,\Lambda,\al} \norm{f-g}_\infty^{p-q} \,.
\ee
If $X$ is further assumed to be geodesic and is such that small enough balls of $X$ are geodesically convex, then for every $k \in \N^*$, and all $p > q >\frac{d(k+1)}{\al}$ 
\be
d_p^p(H_k(X,f),H_k(X,g)) \leq C_{X,\Lambda,\al,k} \norm{f-g}_\infty^{p-q} \,.
\ee
Finally, if $X$ is further supposed to be doubling, then the inequality above holds for all $p>q>\frac{d}{\al}$.
\end{theorem}
\begin{proof}
The first part of the proof is essentially as in \cite{LipschitzStableLpPers}. Start by picking the bottleneck matching between the diagrams of $f$ and $g$ and denote it by $\gamma : \Dgm(f) \to \Dgm(g)$. Then for any $p>q>\frac{d}{\al}$,
\begin{align*}
d_p^p(\Dgm(f),\Dgm(g)) &\leq \sum_{b \in \Dgm(f)} d_{\mathcal{X},\infty}(b, \gamma(b))^p \\
&\leq \norm{f-g}_\infty^{p-q} \sum_{b \in \Dgm(f)} d_{\mathcal{X},\infty}(b, \gamma(b))^{q} \\
&\leq 2^{q}\norm{f-g}_\infty^{p-q} \sum_{b \in \Dgm(f)} d_{\mathcal{X},\infty}(b, \Delta)^{q} + d_{\mathcal{X},\infty}(\gamma(b), \Delta)^{q} \\
&= 2^{q}\norm{f-g}_\infty^{p-q} (\Pers_q^q(f) + \Pers_q^q(g))
\end{align*}
But both $\Pers_q^q(f)$ and $\Pers_q^q(g)$ are bounded above by a global constant for the class $C^\al_\Lambda(X,\R)$, since by the proof of lemma \ref{lemma:PerspvepsContinuity}
\be
N^\veps_f  \leq \mathcal{N}_X\left(\left(\frac{\veps}{2C\Lambda}\right)^{1/\al}\right) \,,
\ee
where $C$ is a constant stemming from the quantitative LLC condition on $X$. This inequality entails that
\begin{align*}
\Pers_q^q(f) = q \int_0^\infty \veps^{q-1} N^\veps_f \; d\veps &\leq q \int_0^{\Lambda \,\text{diam}(X)^\al} \veps^{q-1} \mathcal{N}_X\left(\left(\frac{\veps}{2C\Lambda}\right)^{1/\al}\right) \; d\veps \\
&= (2C\Lambda)^{q} \al q\int_0^{\frac{\text{diam}(X)}{(2C)^{1/\al}}} \veps^{q\al-1} \mathcal{N}_X(\veps) \; d\veps \,,
\end{align*}
which is finite as soon as $q>\frac{d}{\al}$ since $\mathcal{N}_X(\veps) = O(\veps^{-d-\delta})$ as $\veps \to 0$ for all $\delta >0$, by definition of the upper-box dimension. The constant in the statement of the theorem is bounded above by the above estimate. The statements for with the supplementary assumptions of the theorem, the proof follows from the same reasononing by using the proof of theorem \ref{thm:genericityTotalHomology} and remark \ref{rmk:NonDoubling}
\end{proof}
\begin{remark}
More generally, the proof of the theorem adapts with ease to accomodate any compact set of $C^0(X,\R)$ admitting a global modulus of continuity dominated by a H\"older modulus of continuity. It is worth mentioning that such a theorem is impossible to prove for any regularity strictly worse than H\"older, as in such a class of regularity, there are functions $f$ of infinite persistence index, so the theorem is vacuous. 
\end{remark}
Wasserstein stability results are common in the literature and are typically stated by making the following assumption on the underlying metric space $X$. 
\begin{definition}{\cite{LipschitzStableLpPers}}
A metric space \textbf{$X$ implies bounded $q$-total persistence} if, for all $k \in \N$, there exists a constant $C_X$ that depends only on $X$ such that
\be
\Pers_q^q(H_k(X,f)) < C_X 
\ee
for every tame function $f$ with Lipschitz constant $\Lip(f) \leq 1$.
\end{definition}
The regime of validity of Wasserstein stability thus depends solely on this condition on $X$. We can thus see theorem \ref{thm:WassersteinpStabilityLLC} as a theorem giving explicit bounds on the $q$ such that $X$ implies bounded $q$-total persistence (in fact, it does so for every degree in homology independently). Following \cite{LipschitzStableLpPers}, it follows clearly from the proof of Wasserstein stability that this definition implies bounded persistence stability for Lipschitz functions. 
\begin{corollary}
\label{cor:BoundOnCX}
Let $X$ be a compact LLC metric space of $\updim(X) = d$. Then, for all $f\in \Lip_1(X)$ and $p > q>\frac{d}{\al}$,   
\be
\Pers_q^q(H_0(X,f)) \leq (2C\Lambda)^{q} q\int_0^{\frac{\text{diam}(X)}{2C}} \veps^{q-1} \mathcal{N}_X(\veps) \; d\veps  \,.
\ee
If $X$ is further assumed to be geodesic and is such that small enough balls of $X$ are geodesically convex, then for every $k \in \N^*$, and all $p > q >\frac{d(k+1)}{\al}$ 
\be
\Pers_q^q(H_k(X,f))  \leq (2C\Lambda)^{q} q\int_0^{\frac{\text{diam}(X)}{2C}} \veps^{q-1} (\mathcal{N}_X(\veps)\vee K_X)^k \; d\veps \,,
\ee
where $K_X = \mathcal{N}_X(\veps^*)$ and $\veps^*$ is the value after which balls of $X$ are no longer geodesically convex. Finally, if $X$ is further supposed to be $M$-doubling, then for all $p>q>\frac{d}{\al}$,
\be
\Pers_q^q(H_k(X,f))  \leq (2C\Lambda)^{q} q(M^{k+1}-M^k)\int_0^{\frac{\text{diam}(X)}{2C}} \veps^{q-1} (\mathcal{N}_X(\veps)\vee K_X) \; d\veps \,.
\ee
\end{corollary}

Some other Wasserstein $p$ stability results have been reported in the literature: Chen and Edelsbrunner \cite{Chen_2011} studied functions on non-compact domains of $\R^d$, obtaining a stability result which holds for $p>d$. The condition $p>d$ also appears in stability results for \v{C}ech filtrations for point clouds in $\R^d$ and the case of Vietoris-Rips filtrations was recently addressed in \cite{Skraba_2020} by Skraba and Turner.

\subsection{Distance notion and stability for trees}
\begin{definition}
Let $X$ and $Y$ be two compact metric spaces, the \textbf{Gromov-Hausdorff distance, $d_{GH}(X,Y)$} between $X$ and $Y$, is defined as
\be
d_{GH}(X,Y) := \inf_{\substack{f:X \to Z \\ g: Y \to Z}} \max \left\{\sup_{x \in X} \inf_{y\in Y} d_Z(f(x),g(y)), \sup_{y\in Y} \inf_{x \in X} d_Z(f(x),g(y))\right\} \,.
\ee
where the infimum is taken over all metric spaces $Z$ and all isometric embeddings $f: X \to Z$ and $g: Y \to Z$. 
\end{definition}
The Gromov-Hausdorff distance quantifies how far away two metric spaces $X$ and $Y$ are from being isometric to each other. However, it is practically impossible to compute this distance with the above definition. To somewhat alleviate this, we will use the following characterization of the Gromov-Hausdorff distance:
\begin{proposition}[Burago \textit{et al.}, \S 7 \cite{Burago:Metric}]
The Gromov-Hausdorf distance is characterized by 
\be
d_{GH}(X,Y) = \half \inf_{\mathfrak{R}} \sup_{\substack{(x,y) \in \mathfrak{R} \\ (x',y') \in \mathfrak{R}}} \abs{d_X(x,x')-d_Y(y,y')} \,,
\ee
where the infimum is taken over all \textit{correspondences}, \textit{i.e.} subsets $\mathfrak{R} \subset X \times Y$ such that for every $x \in X$ there is at least one $y \in Y$ such that $(x,y) \in \mathfrak{R}$ and a symmetric condition for every $y \in Y$.
\end{proposition}
\begin{remark}
Given two surjective maps $\pi_X : Z \to X$ and $\pi_Y : Z \to Y$, it is possible to build a correspondence between $X$ and $Y$ by considering the set $\{(\pi_X(z), \pi_Y(z)) \in X\times Y \, \vert \, z \in Z\}$.
\end{remark}

A natural question is to ask whether we have an equivalent statement about the stability of $d_{GH}$ with respect to $\norm{\cdot}_{\Linfty}$ and whether the two notions of distances are in some sense ``compatible''. We will positively answer this first question. In general $d_\infty$ and $d_{GH}$ are not compatible, in the sense that no inequality between the two holds in all generality (\cf remark \ref{rmk:incompatibilitydGHdb}). Le Gall and Duquesne \cite{LeGall:Trees} gave a first stability result of $d_{GH}$ with respect to the $L^\infty$-norm on continuous functions on $[0,1]$:
\begin{theorem}[$\Linfty$-stability of trees, \cite{LeGall:Trees}]
\label{thm:stabilityoftrees}
Let $f,g: [0,1] \to \R$ be two continuous functions. Then
\be
d_{GH}(T_f,T_g) \leq 2 \norm{f-g}_{\Linfty} \,.
\ee
\end{theorem}
This result for functions on $[0,1]$ generalizes to more general topological spaces.
\begin{theorem}[Stability theorem for trees]
\label{thm:diagstreesqiso}
Let $X$ be a compact, connected and locally path connected topological space and let $f$ and $g:X\to \R$ be two continuous functions, then
\be
d_{GH}(T_f,T_g) \leq 2 \norm{f-g}_{\Linfty} \,.
\ee
\end{theorem}

\begin{proof}
We will use the distortion characterization of the Gromov-Hausdorff distance, which yields the following inequality
\be
d_{GH}(T_f,T_g) \leq \half \sup_{x,y \in X}\abs{d_f(x,y)-d_g(x,y)} \,.
\ee
Following the logic of the proof of lemma \ref{lemma:fCalpiCal}, the distance between $\pi_f(x)$ and $\pi_f(y)$ is of the form
\be
d_f(\pi_f(x), \tau) + d_f(\tau,\pi_f(y)) = f(x)-f(\tau) + f(y)-f(\tau) 
\ee 
where $\tau$ is the lowest point of the geodesic path in $T_f$ between $\pi_f(x)$ and $\pi_f(y)$. This geodesic path on $T_f$ admits preimages by $\pi_f$ which are paths connecting $x$ to $y$. These paths achieve the following supremum
\be
\sup_{\gamma: x \mapsto y} \inf_{t\in [0,1]} f\circ \gamma = f(\tau) \leq f(x) \wedge f(y)
\ee 
where $a \wedge b := \min\{a,b\}$ since by construction $\gamma$ must always stay above $f(\tau)$ and since for $r>f(\tau)$, $x$ and $y$ lie in different connected components of $X_r$. If $\nu$ is the analogous vertex to $\tau$ on $T_g$,
\begin{align}
d_{GH}(T_f,T_g) &\leq \half \sup_{x,y \in X}\abs{d_f(x,y)-d_g(x,y)} \nonum\\
&= \half \sup_{x,y \in X}\abs{f(x)-g(x) + f(y) - g(y) - 2 f(\tau) + 2 g(\nu)} \nonum\\
&\leq \norm{f-g}_{\Linfty} + \sup_{x,y \in X} \abs{\sup_{\gamma: x \mapsto y} \inf_{t\in [0,1]} f\circ \gamma - \sup_{\eta: x \mapsto y} \inf_{t\in [0,1]} g \circ \eta} \nonum \\
&\leq 2 \norm{f-g}_{\Linfty} \,,
\end{align}
as desired.
\end{proof}

\begin{remark}
\label{rmk:incompatibilitydGHdb}
One can be tempted to establish a general inequality between $d_{GH}$ and $d_\infty$ since both of these distances are bounded by the $\Linfty$-norm. However, this is not possible. 
\par
Indeed, there is a simple counter-example to $d_{GH} \geq d_\infty$. To illustrate this consider two barcodes over a field $k$, $k[s,-\infty[$ and $k[s+\veps,-\infty[$. The bottleneck distance between these two is clearly $\geq \veps$. But supposing that the functions $f$ and $g$ generating these barcodes are such that $f = g+\veps$ the trees $T_g$ and $T_f$ are isometric, so $d_{GH}(T_f,T_g) =0 < \veps \leq d_\infty(\bcode(f),\bcode(g))$.
\par
Conversely, there are also counter-examples to $d_\infty \geq d_{GH}$, as this inequality would imply that two trees which have the same barcode are isometric. This is clearly false, as one can ``glue'' the bars of a given barcode is many different ways to give a tree, which generically will not be isometric. 
\end{remark}

\section{Remarks about stochastic processes}
\label{sec:StochasticProcesses}

As we have previously seen, the study of diagrams of continuous functions involves understanding their regularity. Many stochastic processes are \textit{almost} H\"older continuous in the following sense. 
\begin{definition}
The class of \textbf{almost $\al$-H\"older continuous functions from $X$ to $\R$}, denoted $E^{\al}(X,\R)$ is the class of functions defined by
\be
E^\al(X,\R) := \bigcap_{0 \leq \beta < \al} C^\beta(X,\R)
\ee
\end{definition}
For example, Brownian motion and fractional Brownian motion are in a certain $E^\al$ for some value of $\al$ and moreover, as shown by Kahane \cite[Chapter 7]{Kahane}, random subgaussian Fourier series on torii of any dimension also tend to have $E^\al$ regularities. The ubiquity of $E^\al$-regularities in the context of stochastic processes partially motivate this definition.
\begin{notation}
In what will follow, we will denote $f_\sharp \PP$ the pushforward measure of $\PP$ by $f$.
\end{notation}

\subsection{A change in perspective}
\begin{remark}
Slightly abusing the notation, throughout this section, when we talk about a (continuous) stochastic process, we will talk about a measurable function $f: \Omega \to C^0(X,\R)$ (where $(\Omega,\mathcal{F},\PP)$ is some probability space).
\end{remark}
%Indicate which theorems remain valid for processes defined on LLC metric spaces

Random diagrams, or more precisely, probability measures on the space of diagrams (or on the space of persistence measures) have been studied under many different contexts in the persistence theory literature \cite{ChazalDivol:Brownian,Lacombe_2021,Turner_2014,Fasy_2014,Chazal_2014}. Since ultimately we are interested in studying random processes on some base space $X$, the space of probability measures on the space of diagrams is far too large, as not all diagrams stem from (continuous) functions. In all practical applications, we are never given an abstract persistence diagram. Rather, we compute the persistence diagram from a certain continuous function (on which we may postulate further regularity assumptions, typically that the function is inside some $E^\al(X,\R)$).
This motivates studying subspaces of the full space of persistent diagrams of the form $\cup_k \Dgm_k(E^\al(X,\R)) \subset \mathcal{D}$. This perspective turns out to have notable advantages. For instance, it is known that $(\mathcal{D},d_\infty)$ is not a separable space \cite[Theorem 5]{Bubenik_2018}, but adopting this point of view we can show the opposite.
\begin{proposition}
Let $K \subset (C^0(X,\R),\norm{\cdot}_{\Linfty})$, be a closed subset, then $(\overline{\Dgm(K)},d_\infty)$ is a Polish metric space.
\end{proposition}
\begin{proof}
We start by noticing that the map $\Dgm$ is continuous and that the continuous image of a separable metric space is separable \cite[Theorem 16.4a]{Willard_1970}. Moreover, $\overline{\Dgm(K)}$ remains separable, since the countable dense subset of $\Dgm(K)$ remains dense in the completion.   
\end{proof}
\begin{remark}
If the subset $K$ is compact, then $\Dgm(K) = \overline{\Dgm(K)}$. Notice also that the compact subsets of $C^0(X,\R)$ are sets having a uniform modulus of continuity, by virtue of Ascoli's theorem. In particular, spaces such as $C^\al_\Lambda(X,\R)$ are compact. 
\end{remark}
Consider now continuous $\R$-valued stochastic processes on $X$, $f$, defined on some probability space $(\Omega,\mathcal{F},\PP)$. Then, the space of probability measures on diagrams is also too large, as the probability measures we are concerned with must be of the form $(\Dgm_k \circ f)_\sharp \PP$. For convenience, we could take the closure of this space induced by measures of this form with respect to the topology of vague convergence, or with respect to some Wasserstein distance $W_{p,\delta}$ (on the space of probability measures on diagrams). This is a technical point, but allows us to avoid making hypotheses on the probability measures on the space of diagrams, which are in practice almost never verifiable, and instead give hypotheses on the stochastic processes from which the diagrams stem from.

This point of view is particularly well-suited to look at stochastic processes supported on compact subsets of $C^0(X,\R)$ (in fact, $E^\al(X,\R)$, for reasons which will become apparent later). An easy first result in this direction is that  
\begin{proposition}
Let $K$ be a compact subset of $C^0(X,\R)$, then $\Dgm_k(K) \subset \mathcal{D}_\infty$.
\end{proposition}
This restriction to compact sets can be seen as a considerable limitation. For example, Brownian motion on the interval $[0,1]$ does not satisfy this hypothesis of compactness. However, by virtue of the tightness of probability measures on $C^0(X,\R)$, we may restrict ourselves to a compact $K_\veps$ of $C^0(X,\R)$ in which the process lies with probability $1-\veps$ and make probable statements there, or, alternatively, make conditional statements. 

Furthermore,
\begin{proposition}
\label{prop:EalDp}
Let $(\Omega,\mathcal{F},\PP)$ be a probability space and $f$ be a $\R$-valued, a.s. $E^\al$ stochastic process on a $d$-dimensional compact manifold $X$. Then, for all $\veps>0$, $(\Dgm_k\circ f)_\sharp\PP \in \mathcal{P}(\mathcal{D}_{\frac{d}{\al}+\veps} \cap \mathcal{D}_\infty)$ and \textit{a fortiori} in $\mathcal{P}(\mathcal{D}_r)$ for every $\frac{d}{\al}<r<\infty$. Furthermore, if $\frac{d}{\al}< q < \infty$ and for all $\beta <\al$, $\expect{\norm{f}_{C^\beta(X,\R)}^q} < \infty$, then, $\expect{\Dgm_k(f)} \in \bigcap_{\frac{d}{\al}<p\leq q}\mathcal{D}_p$.
\end{proposition}
\begin{proof}
Since $f \in E^\al(X,\R)$ a.s., it is a.s. $C^\beta(X,\R)$ for every $\beta <\al$, and so a.s. bounded by compactness of $X$. By theorem \ref{thm:genericityTotalHomology} and the previous remark, it follows that for every $k \in \N$, $\Dgm_k(f) \in \mathcal{D}_{\frac{d}{\beta}} \cap \mathcal{D}_\infty$, proving the first result. 

Next, we remark that if $\expect{\norm{f}_{C^{\beta}(X,\R)}^q}$ is finite so is the $p$th moment of the norm for every $1\leq p \leq q$ by a simple application of Jensen's inequality. To show the result, it suffices to show that for such $p$,
\be
\expect{\Pers_p^p(f)} < \infty \,.
\ee
But using the same trick as in the proof of theorem \ref{thm:WassersteinpStabilityLLC}, applying Tonelli's theorem, for some constant $C$ (which is bounded above by the LLC constant of $X$), we have
\begin{align*}
\Pers_p^p(f) &= p \int_0^\infty \veps^{p-1} \expect{N^\veps_f}  \;d\veps \\
&\leq (2C \norm{f}_{C^\beta})^p\beta p \int_0^{\frac{\text{diam}(X)}{(2C)^{1/\beta}}} \veps^{p\beta -1} \mathcal{N}_X(\veps) \; d\veps \,.
\end{align*}
The integral on $[0,1]$ is finite as soon as $p>\frac{d}{\al}$ since the dimension of $X$ is $d$. Taking the expectation of both sides,
\be
\expect{\Pers_p^p(f)} \leq  \tilde{C}_{X,p,\beta} \expect{\norm{f}_{C^\beta}^p} \,,
\ee
which is finite as soon as the moments of the $C^\beta$-norm of $f$ are finite, exactly as supposed in the proposition. Finally, the \textit{a fortiori} inclusion in $\mathcal{D}_r$ is a consequence of the Wasserstein interpolation theorem (proposition \ref{prop:interpolationOT}).
\end{proof}

\subsection{Consequences of stability}
Equipped with some of the elementary facts from optimal transport theory, we may come back to persistence measures and diagrams. The main goal of this section will be to prove the following theorem. 
\begin{theorem}[Stability of random fields under Wasserstein perturbations]
\label{thm:randomfielddiagStability}
Let $f$ and $g$ be two $\R$-valued a.s. $E^\al$ stochastic processes on a $d$ dimensional compact Riemannian manifold $X$ on a probability space $(\Omega,\mathcal{F},\PP)$. Then, for any $k \in \N$ and any $1 \leq p \leq \infty$,
\be
W_{p,d_\infty}((\Dgm_k \circ f)_\sharp \PP, (\Dgm_k \circ g)_\sharp \PP) \leq W_{p,L^\infty}(f_\sharp \PP, g_\sharp \PP) \,.
\ee 
Moreover, if the supports of $f_\sharp \PP$ and $g_\sharp \PP$ are compact in $E^\al(X,\R)$, then 
\be
d_\infty(\expect{\Dgm_k(f)},\expect{\Dgm_k(g)}) \leq  W_{\infty,d_\infty}((\Dgm_k \circ f)_\sharp \PP, (\Dgm_k \circ g)_\sharp \PP)   \,,
\ee  
and for every $\frac{d}{\al}<p\leq q \leq \infty$, there exists a constant $C_{X,p,\eta}$ depending on the supports of $f_\sharp \PP$ and $g_\sharp \PP$ such  that 
\be
d_p(\expect{\Dgm_k(f)},\expect{\Dgm_k(g)}) \leq W_{q,d_p}((\Dgm_k \circ f)_\sharp\PP,(\Dgm_k \circ g)_\sharp\PP ) \leq C_{X,p,\eta} W_{q\eta,L^\infty}^\eta(f_\sharp \PP, g_\sharp \PP) \,.
\ee
where $\eta < 1- \frac{d}{\al p}$.
\end{theorem}
\begin{remark}
The proof of this theorem uses some of the techniques from \cite[Lemma 15]{Chazal_2014}. It differs from this result, as it concerns the $d_p$-stability as opposed to simply $d_\infty$-stability, but also because the statement of theorem \ref{thm:randomfielddiagStability} gives a bound on the distance between expected diagrams, as opposed to a linear functional of the latter. However, necessary and sufficient conditions for the continuity of linear functionals of $\expect{\Dgm(f)} \in (\DD_p, d_p)$ has been studied by Divol and Lacombe in \cite{Divol_2019}. 
\end{remark}
\begin{proof}[Proof of theorem \ref{thm:randomfielddiagStability}]
The first inequality is a simple consequence of a change of variables and an application of the bottleneck stability theorem. Next, notice that if $f_\sharp \PP$ and $g_\sharp \PP$ have compact support in $E^\al$, then $f$ and $g$ are almost surely uniformly bounded functions, so $\expect{\Dgm(f)}$ and $\expect{\Dgm(g)}$ are both in $\mathcal{D}_\infty$. 

Notice that,
\be
\expect{\Dgm(f)} = \int_{E^\al} \Dgm(h) \;df_\sharp\PP(h) = \int_{(E^\al)^2} \Dgm(h) \; d\pi(h,\tilde{h}) \,,
\ee
for any $\pi \in \Gamma(f_\sharp \PP, g_\sharp \PP)$ and an analogous equality holds for $\expect{\Dgm(g)}$. Since $d_p^p$ is convex, applying Jensen's inequality 
\begin{align*}
d_p^p(\expect{\Dgm(f)},\expect{\Dgm(g)}) &= d_p^p\left(\int_{(E^\al)^2} \Dgm(h) \; d\pi(h,\tilde{h}), \int_{(E^\al)^2} \Dgm(\tilde{h}) \; d\pi(h,\tilde{h})\right) \\
&\leq \int_{(E^\al)^2} d_p^p(\Dgm(h),\Dgm(\tilde{h})) \; d\pi(h,\tilde{h}) \\
&= \int_{(\Dgm(E^\al))^2} d_p^p(x,y) \;d\!\Dgm^{\tensor 2}_\sharp \pi(x,y) \,.
\end{align*}
Taking the infimum over every $\pi$ of this inequality and taking the $p$th root,
\begin{align*}
d_p(\expect{\Dgm(f)},\expect{\Dgm(g)}) \leq W_{p,d_p}((\Dgm\circ f)_\sharp \PP, (\Dgm\circ g)_\sharp \PP) \,.
\end{align*}
The result for $p=\infty$ is obtained by taking the limit $p\to \infty$, justified by remark \ref{rmk:limitptoinfty} and the fact that the stochastic processes and their distributions in $E^\al$ are uniformly bounded. Keeping the same notation, if $\pi$ is an optimal transport for $d_p$, $\pi$ must necessarily be itself of compact support within $(E^\al)^2$. In particular, for any $\beta <\al$, if $K_f$ and $K_g$ denote the supports of $f_\sharp \PP$ and $g_\sharp\PP$, there exists a finite constant 
\be
\Lambda := \left(\sup_{\vp \in K_f} \norm{\vp}_{C^\beta}\right) \vee \left(\sup_{\psi \in K_g} \norm{\psi}_{C^\beta}\right) \,,
\ee
such that, applying the Wasserstein $p$ stability theorem for all $p > k> \frac{d}{\beta}$,
\be
\int_{(E^\al)^2} d_p^p(\Dgm(h),\Dgm(\tilde{h})) \; d\pi(h,\tilde{h}) \leq C_{X,\Lambda,\beta} \int_{(E^\al)^2}\norm{h-\tilde{h}}_\infty^{p-k} d\pi(h,\tilde{h}) \,,
\ee
yielding the result of the theorem for the values of $\eta$ prescribed. That the same inequalities hold for all $p\leq q \leq \infty$ is a consequence of Jensen's inequality.
\end{proof}

\begin{proposition}[Control of $W_{p,L^\infty}$]
\label{prop:ControlofWpLinfty}
Let $f$ and $g$ be two $\R$-valued a.s. $E^\al$ stochastic processes on a $d$ dimensional compact Riemannian manifold $X$ on a probability space $(\Omega,\mathcal{F},\PP)$. Then, the following inequality holds
\be
W_{p,L^\infty}(f_\sharp\PP, g_\sharp \PP) \leq \norm{f-g}_{L^p(\Omega,L^\infty(X,\R))}
\ee
\end{proposition}
\begin{proof}
The map $F : \Omega \to E^\al(X,\R)^2$ which sends $\omega \mapsto (f(\omega),g(\omega))$ induces a transport map $F_\sharp\PP \in \Gamma(f_\sharp\PP,g_\sharp\PP)$ and 
\begin{align*}
W_{p,L^\infty}^p(f_\sharp\PP,g_\sharp \PP) &\leq \int_{E^\al(X,\R)^2 } \norm{h-k}^p dF_\sharp \PP(h,k) = \int_\Omega \norm{f(\omega)-g(\omega)}_\infty^p d\PP(\omega)\\
&=\norm{f-g}_{L^p(\Omega,L^\infty(X,\R))}^p \,,
\end{align*}
which finishes the proof.
\end{proof}
\begin{remark}
Proposition \ref{prop:ControlofWpLinfty} yields an easy way to estimate the value of Wasserstein distances between stochastic processes. Using the results of \cite{Perez_Pr_2020} and other results on rates of convergence of random processes (which could be obtained by using results such as those of Kahane \cite{Kahane}), this instantly gives estimates for Wasserstein distances between distributions for a panoply of processes.
\end{remark}
\begin{corollary}[A remark on discretization]
Keeping the same notation, fix a triangulation $P$ of $X$ whose $0$-skeleton has $n$ points and such that the $0$-skeleton of $P$ is an $\veps$-net of $X$ (this constrains $n \geq \mathcal{N}_X(\veps)$) and define a new process $\hat{f}$ which is equal to $f$ on the $0$-skeleton of $P$ and linearly interpolate in between. Then,
\be
W_{p,L^\infty}(f_\sharp \PP, \hat{f}_\sharp \PP) \leq \expect{\norm{f}_{C^\beta}^p} \veps^{\beta p} \,.
\ee
If $p=\infty$ and that $\norm{f}_{C^\beta}$ is uniformly bounded by $L$, then
\be
W_{\infty,L^\infty}(f_\sharp\PP,\hat{f}_\sharp \PP) \leq L\veps^\al
\ee
and theorem \ref{thm:randomfielddiagStability} applies.
\end{corollary}
\begin{proof}
 Clearly, $\hat{f} : \Omega \to \Lip_{\Lambda_\veps}(X,\R)$ of law $\hat{f}_\sharp\PP$. By proposition \ref{prop:ControlofWpLinfty}, for any $\beta < \al$, 
\begin{align*}
W_{p,L^\infty}^p(f_\sharp\PP,\hat{f}_\sharp \PP) &\leq \expect{\norm{f-\hat{f}}_\infty^p} \leq \expect{\norm{f}_{C^\beta}^p} \veps^{\beta p} \,.
\end{align*}
Taking $p \to \infty$, provided that the distribution of $\norm{f}_{C^\beta}$ has bounded support, we can bound the support of this distribution by $L$, we get $W_{\infty,L^\infty}(f_\sharp\PP,\hat{f}_\sharp \PP) \leq L\veps^\al$. In particular, the expected diagrams differ from less than $L\veps^\al$ in $d_\infty$.
\end{proof}

\begin{remark}
The topology on the measures on $C^0(X,\R)$ defined by Wasserstein distances may be too weak. Indeed, note that $W_{p,\Linfty}$-balls around any measure $\mu$ supported on some $E^\al(X,\R)$ include probability measures whose support intersects sets of $C^0(X,\R)$ whose number of small bars grows faster than any polynomial (or indeed any computable function!). To see why, it suffices to exhibit an example of such a function (let us denote it $h$), and notice that if a stochastic process $f$ has law $\mu$, if $\xi$ denotes a standard gaussian random variable, then $f + \veps \xi h$ is (up to rendering $f$ locally constant on some small ball) an arbitrarily small $\Linfty$-perturbation of $f$ whose number of small bars grows arbitrarily fast. In particular, this perturbation is not in any $\mathcal{D}_p$ for any $p$, but the law of this perturbed process is included within a $W_{p,\Linfty}$-ball of arbitrarily small radius. 

However, by changing topology to that of a Sobolev space which injects itself onto some $C^\al(X,\R)$, we can avoid this problem. With this change in topology, it might be superfluous to require that the processes lie in $E^\al(X,\R)$, as it might follow from an argument ressembling that of the proof of the Kolmogorov-Chentsov theorem (theorem \ref{thm:KolmogorovChentsov}). 
\end{remark}
% The previous remark shows that it is impossible that expected Betti curves or Euler curves be stable under $W_{p,\Linfty}$-perturbations of the underlying distributions of stochastic processes. This can also be shown 

\subsection{Establishing classes of regularity}
A sufficient and easily verifiable condition for a stochastic process to be almost surely $E^\al$ is given by the Kolmogorov-Chentsov theorem.
\begin{theorem}[Kolmogorov-Chentsov Theorem for compact manifolds, \cite{Andreev_2014,Aubin_1998}]
\label{thm:KolmogorovChentsov}
Let $(\Omega, \mathcal{F},\PP)$ be a probability space, $\mathcal{B}$ be a Banach space, $X$ be a $d$-dimensional compact Riemannian manifold (without boundary) with distance $d_X$ and $f:\Omega \times X \to \mathcal{B}$ be a $\mathcal{B}$-valued separable stochastic process. Suppose there exists constants $C>0$, $\veps>0$ and $\delta>1$ such that for all $x,y \in X$,
\be
\expect{\norm{f(x)-f(y)}_\mathcal{B}^\delta} \leq C d_X(x,y)^{d+\veps} \,,
\ee 
then there exists a modification of $f$ such that for all $\al \in [0,\frac{\veps}{\delta}[$, $f$ is almost surely $\al$-H\"older continuous.
\end{theorem}
The proof uses the same idea of \cite{Andreev_2014} to use the Sobolev embedding theorem. For compact Riemannian manifolds, the required Sobolev embedding theorem is given by \cite[Theorem 2.20]{Aubin_1998} (in fact, within \cite{Aubin_1998}, one can actually find Sobolev embedding theorems valid for wider classes of manifolds). Let us give a sketch of the proof.
\begin{proof}[Sketch of proof of theorem \ref{thm:KolmogorovChentsov}]
First, by virtue of Markov's inequality, the estimation on the moments above entails that the process is continuous in probability. We may therefore assume that, up to taking a modification of $f$, the process $f$ is measurable on $\Omega \times X$. Fix $\gamma$ a real number, then Tonelli's theorem and the estimation of the moments above implies that
\begin{align*}
\expect{\int_X \int_X \frac{\norm{f(x)-f(y)}_{\mathcal{B}}^\delta}{d_X(x,y)^{d+\gamma\delta}} \; dx \, dy} &= \int_X \int_X \frac{\expect{\norm{f(x)-f(y)}_{\mathcal{B}}^\delta}}{d_X(x,y)^{d+\gamma\delta}} \; dx \, dy\\
&\leq C\int_X \int_X d_X(x,y)^{\veps - \gamma\delta}\; dx \, dy 
\end{align*}
which is finite as soon as $\gamma < \frac{d+\veps}{\delta}$. Notice that the bounded quantity is nothing other than the norm of $f$ in $L^\delta(\Omega,W^{\gamma,\delta}(X,\mathcal{B}))$, so that almost surely, $f_\omega \in W^{\gamma,\delta}(X,\mathcal{B})$. There is a Sobolev injection of $W^{\gamma,\delta}(X,\mathcal{B}) \xhookrightarrow{} C^\al(X,\mathcal{B})$ for all $\al < \gamma - \frac{d}{\delta}$, so for every $\al < \frac{\veps}{\delta}$, there is a measurable set $\Omega_0 \subset \Omega$ of probability measure $1$ on which for every $\omega \in \Omega_0$,  $f_\omega$ is $\al$-H\"older almost everywhere on $X$. The corresponding modification can be obtained by making the trajectories continuous everywhere. Since the process $f$ is measurable on $\Omega \times X$, we can set
\be
g_{\omega}(h,x) := \frac{1}{\Vol(B(x,h))} \int_{B(x,h)} f_\omega(y) \;dy \,,
\ee
and consider the set 
\be
B = \{(\omega,x) \in \Omega \times X \,\vert \, (g_\omega(h,x))_h \text{ converges as } h \to 0\}
\ee
and set the continuous modification of $f$ to be
\be
g_\omega(x):= \begin{cases} \lim_{h \to 0} g_\omega(h,x) & (\omega,x) \in B \\ 0 & \text{else} \end{cases} \,.
\ee
Finally, it is easy to check this function is indeed $\al$-H\"older everywhere on $\Omega_0$ and to check that $\PP(g(x)=f(x))=1$ almost everywhere on $X$.
\end{proof}
\begin{remark}
If $\mathcal{B} = \R$, the same idea works (as shown in \cite{Andreev_2014}) to prove results on the existence of modifications of processes such that the modification is almost surely of class $C^k$. 
\end{remark}
Provided that we have control over all moments of $\norm{f(x)-f(y)}$, the Kolmogorov-Chentsov theorem constrains the regularity of the process to live within some family
\be
\bigcap_{0 \leq \al < \al^*} C^\al(X,\R)
\ee 
for some $\al^*$. As an immediate corollary,
\begin{corollary}
With the same hypotheses and notation of theorem \ref{thm:KolmogorovChentsov} where now $\mathcal{B}=\R$, denoting $\al^*:=\sup_{\veps,\delta} \frac{\veps}{\delta}$, almost surely,
\be
\Lag_{Tot}(f) \leq \frac{d}{\al^*} \,.
\ee
\end{corollary}

\section{Acknowledgements}
The author would like to thank Pierre Pansu and Claude Viterbo for helping with the redaction of the manuscript as well as their guidance. Many thanks are also owed to Shmuel Weinberger, Yuliy Baryshnikov and Jean-Fran{\c c}ois Le Gall and Nicolas Curien for the fruitful discussions without which some of this work would not have been possible. 

\bibliographystyle{abbrv}
\bibliography{PhDThesis}

\begin{thebibliography}{10}

\bibitem{Adams_2020}
H.~Adams, M.~Aminian, E.~Farnell, M.~Kirby, J.~Mirth, R.~Neville, C.~Peterson,
  and C.~Shonkwiler.
\newblock A fractal dimension for measures via persistent homology.
\newblock {\em Abel Symposia}, pages 1--31, 2020.

\bibitem{Adams_2017}
H.~Adams, T.~Emerson, M.~Kirby, R.~Neville, C.~Peterson, P.~Shipman,
  S.~Chepushtanova, E.~Hanson, F.~Motta, and L.~Ziegelmeier.
\newblock Persistence images: A stable vector representation of persistent
  homology.
\newblock {\em Journal of Machine Learning Research}, 18(8):1--35, 2017.

\bibitem{Adler_2010}
R.~J. Adler, O.~Bobrowski, M.~S. Borman, E.~Subag, and S.~Weinberger.
\newblock Persistent homology for random fields and complexes.
\newblock In {\em Institute of Mathematical Statistics Collections}, pages
  124--143. Institute of Mathematical Statistics, 2010.

\bibitem{AdlerTaylor:RandomFields}
R.~J. Adler and J.~E. Taylor.
\newblock {\em Random Fields and Geometry}.
\newblock Springer New York, 2007.

\bibitem{Andreev_2014}
R.~Andreev and A.~Lang.
\newblock {Kolmogorov-Chentsov} theorem and differentiability of random fields
  on manifolds.
\newblock {\em Potential Analysis}, 41(3):761--769, feb 2014.

\bibitem{Aubin_1998}
T.~Aubin.
\newblock {\em Some Nonlinear Problems in Riemannian Geometry}.
\newblock Springer Berlin Heidelberg, 1998.

\bibitem{Baryshnikov_2019}
Y.~Baryshnikov.
\newblock Time series, persistent homology and chirality.
\newblock {\em arXiv:1909.09846}, 2019.

\bibitem{Boardman_1968}
J.~M. Boardman and R.~M. Vogt.
\newblock Homotopy-everything $h$-spaces.
\newblock {\em Bulletin of the American Mathematical Society},
  74(6):1117--1123, nov 1968.

\bibitem{Boardman_1973}
J.~M. Boardman and R.~M. Vogt.
\newblock {\em Homotopy Invariant Algebraic Structures on Topological Spaces}.
\newblock Springer Berlin Heidelberg, 1973.

\bibitem{Bubenik_2018}
P.~Bubenik and T.~Vergili.
\newblock Topological spaces of persistence modules and their properties.
\newblock {\em Journal of Applied and Computational Topology}, 2(3-4):233--269,
  dec 2018.

\bibitem{Burago:Metric}
D.~Burago, Y.~Burago, and S.~Ivanov.
\newblock {\em A course in metric geometry}, volume~33 of {\em Graduate Studies
  in Mathematics}.
\newblock American Mathematical Society, Providence, RI, 2001.

\bibitem{Carriere_2017}
M.~Carriere, S.~Oudot, and M.~Ovsjanikov.
\newblock {Sliced Wasserstein Kernel for Persistence Diagrams}.
\newblock In {\em {ICML 2017 - Thirty-fourth International Conference on
  Machine Learning}}, pages 1--10, Sydney, Australia, Aug. 2017.

\bibitem{Chazal:Persistence}
F.~Chazal, V.~de~Silva, M.~Glisse, and S.~Oudot.
\newblock {\em The Structure and Stability of Persistence Modules}.
\newblock Springer International Publishing, 2016.

\bibitem{ChazalDivol:Brownian}
F.~Chazal and V.~Divol.
\newblock The density of expected persistence diagrams and its kernel based
  estimation.
\newblock In B.~Speckmann and C.~D. T{\'o}th, editors, {\em 34th International
  Symposium on Computational Geometry (SoCG 2018)}, volume~99 of {\em Leibniz
  International Proceedings in Informatics (LIPIcs)}, pages 26:1--26:15,
  Dagstuhl, Germany, 2018. Schloss Dagstuhl--Leibniz-Zentrum fuer Informatik.

\bibitem{Chazal_2015}
F.~Chazal, B.~Fasy, F.~Lecci, B.~Michel, A.~Rinaldo, and L.~Wasserman.
\newblock Subsampling methods for persistent homology.
\newblock In F.~Bach and D.~Blei, editors, {\em Proceedings of the 32nd
  International Conference on Machine Learning}, volume~37 of {\em Proceedings
  of Machine Learning Research}, pages 2143--2151, Lille, France, 07--09 Jul
  2015. PMLR.

\bibitem{Chazal_2014}
F.~Chazal, M.~Glisse, C.~Labru{\`e}re, and B.~Michel.
\newblock Convergence rates for persistence diagram estimation in topological
  data analysis.
\newblock In E.~P. Xing and T.~Jebara, editors, {\em Proceedings of the 31st
  International Conference on Machine Learning}, volume~32 of {\em Proceedings
  of Machine Learning Research}, pages 163--171, Bejing, China, 22--24 Jun
  2014. PMLR.

\bibitem{Chen_2011}
C.~Chen and H.~Edelsbrunner.
\newblock Diffusion runs low on persistence fast.
\newblock In {\em 2011 International Conference on Computer Vision}, pages
  423--430, 2011.

\bibitem{Chiswell_2001}
I.~Chiswell.
\newblock {\em Introduction to $\Lambda$-Trees}.
\newblock {World Scientific}, feb 2001.

\bibitem{Chizat_2015}
L.~Chizat, G.~Peyr{\'e}, B.~Schmitzer, and F.-X. Vialard.
\newblock Unbalanced optimal transport: Dynamic and kantorovich formulation,
  2015.

\bibitem{LipschitzStableLpPers}
D.~Cohen-Steiner, H.~Edelsbrunner, J.~Harer, and Y.~Mileyko.
\newblock Lipschitz functions have {$L^p$}-stable persistence.
\newblock {\em Foundations of Computational Mathematics}, 10(2):127--139, Jan
  2010.

\bibitem{curien2013}
N.~Curien, J.-F. Le~Gall, and G.~Miermont.
\newblock {The Brownian cactus I. Scaling limits of discrete cactuses}.
\newblock {\em Ann. Inst. H. Poincar{\'e} Probab. Statist.}, 49(2):340--373, 05
  2013.

\bibitem{Curry_2018}
J.~Curry.
\newblock The fiber of the persistence map for functions on the interval.
\newblock {\em Journal of Applied and Computational Topology}, 2(3-4):301--321,
  dec 2018.

\bibitem{Curry_2021}
J.~Curry, H.~Hang, W.~Mio, T.~Needham, and O.~B. Okutan.
\newblock Decorated merge trees for persistent topology, 2021.

\bibitem{Divol_2019}
V.~Divol and T.~Lacombe.
\newblock Understanding the topology and the geometry of the persistence
  diagram space via optimal partial transport.
\newblock {\em CoRR}, abs/1901.03048, 2019.

\bibitem{Lacombe_2021}
V.~Divol and T.~Lacombe.
\newblock Estimation and quantization of expected persistence diagrams.
\newblock 2021.

\bibitem{Polonik_2019}
V.~Divol and W.~Polonik.
\newblock On the choice of weight functions for linear representations of
  persistence diagrams.
\newblock {\em Journal of Applied and Computational Topology}, 3(3):249--283,
  aug 2019.

\bibitem{LeGall:Trees}
T.~Duquesne and J.-F. Le~Gall.
\newblock {\em Random trees, {L{\'e}vy} processes and spatial branching
  processes}.
\newblock Number 281 in Ast{\'e}risque. Soci{\'e}t{\'e} math{\'e}matique de
  France, 2002.

\bibitem{LeGallDuquesne:LevyTrees}
T.~Duquesne and J.-F. Le~Gall.
\newblock Probabilistic and fractal aspects of {L{\'e}vy} trees.
\newblock {\em Probability Theory and Related Fields}, 131(4):553--603, Nov
  2004.

\bibitem{HarerEdelsbrunner:CT}
H.~Edelsbrunner and J.~Harer.
\newblock {\em Computational Topology: An Introduction}.
\newblock 01 2010.

\bibitem{Eilenberg_1952}
S.~Eilenberg and N.~Steenrod.
\newblock {\em Foundations of Algebraic Topology}.
\newblock Princeton University Press, dec 1952.

\bibitem{Fasy_2014}
B.~T. Fasy, F.~Lecci, A.~Rinaldo, L.~Wasserman, S.~Balakrishnan, and A.~Singh.
\newblock Confidence sets for persistence diagrams.
\newblock {\em The Annals of Statistics}, 42(6), dec 2014.

\bibitem{Figalli_2009}
A.~Figalli.
\newblock The optimal partial transport problem.
\newblock {\em Archive for Rational Mechanics and Analysis}, 195(2):533--560,
  jan 2009.

\bibitem{FIGALLI2010107}
A.~Figalli and N.~Gigli.
\newblock A new transportation distance between non-negative measures, with
  applications to gradients flows with dirichlet boundary conditions.
\newblock {\em Journal de Math{\'e}matiques Pures et Appliqu{\'e}es},
  94(2):107--130, 2010.

\bibitem{Ginzburg_1994}
V.~Ginzburg and M.~Kapranov.
\newblock Koszul duality for operads.
\newblock {\em Duke Mathematical Journal}, 76(1):203--272, oct 1994.

\bibitem{Givens_1984}
C.~R. Givens and R.~M. Shortt.
\newblock A class of {Wasserstein} metrics for probability distributions.
\newblock {\em Michigan Mathematical Journal}, 31(2), jan 1984.

\bibitem{Jiang_2021}
H.~Jiang and Y.-H. Yang.
\newblock Manifolds of positive ricci curvature with quadratically
  asymptotically nonnegative curvature and infinite topological type.
\newblock {\em Communications in Analysis and Geometry}, 29(5):1233--1253,
  2021.

\bibitem{Kahane}
J.-P. Kahane.
\newblock {\em Some random series of functions}.
\newblock Cambridge University Press, Feb 1986.

\bibitem{Kondratyev_2016}
S.~{Kondratyev}, L.~{Monsaingeon}, and D.~{Vorotnikov}.
\newblock {A new optimal transport distance on the space of finite Radon
  measures}.
\newblock {\em {Adv. Differ. Equ.}}, 21(11-12):1117--1164, 2016.

\bibitem{Kozma_2006}
G.~Kozma, Z.~Lotker, and G.~Stupp.
\newblock The minimal spanning tree and the upper box dimension.
\newblock {\em Proceedings of the American Mathematical Society},
  134(4):1183--1187, 2006.

\bibitem{Loday_1996}
J.-L. {Loday}.
\newblock {La renaissance des op\'erades}.
\newblock In {\em {S\'eminaire Bourbaki. Volume 1994/95. Expos\'es 790-804}},
  pages 47--74, ex. Paris: Soci\'et\'e Ma\-th\'e\-ma\-tique de France, 1996.

\bibitem{MacPherson_2012}
R.~MacPherson and B.~Schweinhart.
\newblock Measuring shape with topology.
\newblock {\em Journal of Mathematical Physics}, 53(7):073516, Jul 2012.

\bibitem{May_1972}
J.~P. May.
\newblock {\em The Geometry of Iterated Loop Spaces}.
\newblock Springer Berlin Heidelberg, 1972.

\bibitem{Memoli_2020}
F.~M{\'e}moli and O.~B. Okutan.
\newblock Reeb posets and tree approximations.
\newblock {\em Discrete Mathematics}, 343(2):111658, 2020.

\bibitem{Menguy_2001}
X.~Menguy.
\newblock Examples of strictly weakly regular points.
\newblock {\em Geometric and Functional Analysis}, 11(1):124--131, apr 2001.

\bibitem{Menguy_Thesis}
X.~C. Menguy.
\newblock {\em {Examples of manifolds and spaces with positive Ricci
  curvature}}.
\newblock PhD thesis, New York University, 2000.

\bibitem{Mileyko_2011}
Y.~Mileyko, S.~Mukherjee, and J.~Harer.
\newblock Probability measures on the space of persistence diagrams.
\newblock {\em Inverse Problems}, 27(12):124007, Nov 2011.

\bibitem{Munch_2019}
E.~Munch and A.~Stefanou.
\newblock The $\ell$ $\infty$-cophenetic metric for phylogenetic trees as an
  interleaving distance.
\newblock In {\em Association for Women in Mathematics Series}, pages 109--127.
  Springer International Publishing, 2019.

\bibitem{Oudot:Persistence}
S.~Y. Oudot.
\newblock {\em Persistence Theory - From Quiver Representations to Data
  Analysis}, volume 209 of {\em Mathematical surveys and monographs}.
\newblock American Mathematical Society, 2015.

\bibitem{Perez_Pr_2020}
D.~{Perez}.
\newblock {On the persistent homology of almost surely $C^0$ stochastic
  processes}.
\newblock \url{https://arxiv.org/abs/2012.09459}, Dec. 2020.

\bibitem{Picard:Trees}
J.~Picard.
\newblock A tree approach to $p$-variation and to integration.
\newblock {\em The Annals of Probability}, 36(6):2235--2279, Nov 2008.

\bibitem{Polterovitch:LaplaceEigenfunctions}
I.~Polterovich, L.~Polterovich, and V.~Stojisavljevi{\'c}.
\newblock Persistence barcodes and {Laplace} eigenfunctions on surfaces.
\newblock {\em Geometriae Dedicata}, 201(1):111--138, Aug 2018.

\bibitem{Polterovitch:Persistence}
L.~{Polterovich}, D.~{Rosen}, K.~{Samvelyan}, and J.~{Zhang}.
\newblock {Topological Persistence in Geometry and Analysis}.
\newblock {\em arXiv e-prints}, page arXiv:1904.04044, Apr 2019.

\bibitem{Schweinhart_2019}
B.~Schweinhart.
\newblock Persistent homology and the upper box dimension.
\newblock {\em {Discrete \& Computational Geometry}}, Nov 2019.

\bibitem{Skraba_2020}
P.~Skraba and K.~Turner.
\newblock Wasserstein stability for persistence diagrams, 2020.

\bibitem{Stanley_2009}
R.~P. Stanley.
\newblock {\em Enumerative Combinatorics}.
\newblock Cambridge University Press, 2009.

\bibitem{Turner_2014}
K.~Turner, Y.~Mileyko, S.~Mukherjee, and J.~Harer.
\newblock {Fr{\'e}chet} means for distributions of persistence diagrams.
\newblock {\em Discrete \& Computational Geometry}, 52(1):44--70, Jul 2014.

\bibitem{Wang_2014}
S.~Wang, Y.~Wang, and R.~Wenger.
\newblock The {JS}-graphs of join and split trees.
\newblock In {\em Proceedings of the thirtieth annual symposium on
  Computational geometry}. {ACM}, jun 2014.

\bibitem{Willard_1970}
S.~Willard.
\newblock {\em General topology}.
\newblock Addison-Wesley Publishing Co., 1970.

\end{thebibliography}

\end{document}